\documentclass[12pt]{amsart}
\pdfoutput=1

\usepackage{amsmath,amssymb,amsthm}
\usepackage{mathtools}
\usepackage{thm-restate}
\usepackage{tikz}
\usetikzlibrary{fadings}
\usetikzlibrary{decorations.pathreplacing}
\usepackage{comment}
\usepackage{enumitem}

\usepackage{graphicx}              
\usepackage{caption}
\usepackage{subcaption}
\usepackage{graphics}

\usepackage{tikz}
\usetikzlibrary{arrows}
\tikzset{edge/.style = {->,> = stealth'}}
\usetikzlibrary{angles,quotes}
\usetikzlibrary{decorations.pathreplacing}
\definecolor{mygray}{RGB}{160,160,160}

\usepackage{xcolor}
\definecolor{CornflowerBlue}{rgb}{0.39, 0.58, 0.93}
\definecolor{LavenderMagenta}{rgb}{0.93, 0.51, 0.93}
\definecolor{PastelOrange}{rgb}{1.0, 0.7, 0.28}

\usepackage{xcolor}
\usepackage[unicode]{hyperref}
\hypersetup{
	colorlinks,
	linkcolor={red!60!black},
	citecolor={green!60!black},
	urlcolor={blue!60!black},
}
\usepackage[abbrev, msc-links]{amsrefs}
\usepackage[capitalize, nameinlink]{cleveref}

\usepackage[utf8]{inputenc}
\usepackage[T1]{fontenc}
\usepackage{lmodern}
\usepackage[babel]{microtype}
\usepackage[english]{babel}

\usepackage{geometry}
\usepackage[presets={vec-cev}]{letterswitharrows}

\usepackage{enumitem}

\let\polishlcross=\l
\def\l{\ifmmode\ell\else\polishlcross\fi}

\let\emptyset=\varnothing

\let\theta=\vartheta
\let\rho=\varrho
\let\phi=\varphi

\def\NN{\mathbb N}

\newcommand{\Set}[1]{{\left\lbrace {#1} \right\rbrace}}

\def\set#1:#2{\Set{{#1} \colon {#2}}}

\theoremstyle{plain}
\newtheorem{thm}{Theorem}[section]

\newtheorem{prop}[thm]{Proposition}
\newtheorem{cor}[thm]{Corollary}
\newtheorem{lem}[thm]{Lemma}

\newtheorem{obs}[thm]{Observation}

\theoremstyle{definition}

\title{A star-comb lemma for infinite digraphs}
\author{Florian Reich}
\address{Universit\"at Hamburg, Department of Mathematics, Bundesstrasse 55 (Geomatikum), 20146 Hamburg, Germany}
\email{florian.reich@uni-hamburg.de}

\keywords{connectivity, infinite digraph, star, comb}

\begin{document}

\begin{abstract}
	The star-comb lemma is a standard tool in infinite graph theory, which states that for every infinite set $U$ of vertices in a connected graph $G$ there exists either a subdivided infinite star in $G$ with all leaves in $U$, or an infinite comb in $G$ with all teeth in $U$.
	
	In this paper, we elaborate a counterpart of the star-comb lemma for directed graphs.
	More precisely, we prove that for every infinite set $U$ of vertices in a strongly connected directed graph~$D$, there exists a strongly connected butterfly minor of $D$ with infinitely many teeth in $U$ that is either shaped by a star or shaped by a comb, or is a `chain of triangles'.
\end{abstract}

\maketitle

\section{Introduction}
K\H{o}nig's Infinity Lemma~\cite{konig1927schlussweise} shows that
\emph{stars}, that is graphs $K_{1,\infty}$, and rays, that is one-way infinite paths, are unavoidable subgraphs of infinite connected graphs:
\vspace{.5\baselineskip}
\begin{equation}
	\begin{aligned}\label{Fact:infinite_undirected_simple}
		\parbox{\textwidth-3\parindent}{\emph{
				Every infinite connected graph contains either a star or a ray.}}
	\end{aligned}
	\tag{$\ast$}
\end{equation}

\vspace{.5\baselineskip}
\noindent
The local variant of~\labelcref{Fact:infinite_undirected_simple}, known as the star-comb lemma, is a standard tool in infinite graph theory:
\begin{restatable}{starCombIntro}{undirectedStarComb}  \cite{diestel}*{Lemma 8.2.2}
	For every infinite set $U$ of vertices of a connected graph~$G$ there exists either a subdivided star or a comb in $G$ that has all teeth in $U$.
\end{restatable}
\noindent
A \emph{comb} is the union of a ray $R$ with disjoint (possibly trivial) paths having precisely their first vertex on $R$ and we refer to the last vertices of these paths as the \emph{teeth} of this comb.
We also refer to the leaves of a star as its \emph{teeth}.
Furthermore, we call a vertex $v$ of a comb a \emph{junction} if $v$ has degree $3$ or $v$ is a tooth and has degree $2$.

The unavoidable induced subgraphs of infinite connected graphs are stars, rays and graphs $K_\infty$, that is countably infinite complete graphs, which is an immediate consequence of Ramsey's theorem~\cite{ramsey} and K\H{o}nig's Infinity Lemma~\cite{konig1927schlussweise}.
The study of unavoidable substructures in infinite graphs extends to higher connectivity:
Halin~\cite{halin1978simplicial} studied unavoidable topological minors in uncountably infinite $k$-connected graphs for arbitrary $k \in \NN$.
Oporowski, Oxley and Thomas~\cite{oporowski1993typical} extended Halin's result to minors in countably infinite essentially $k$-connected graphs.
Furthermore, the unavoidable induced subgraphs in infinite $2$-connected graphs were discovered by Allred, Ding and Oporowski~\cite{allred2025unavoidable}.
Recent progress on unavoidable topological minors in infinite $2$- and $3$-connected rooted graphs was made by Nguyen~\cite{nguyen2024unavoidable}.

In the context of directed graphs, the unavoidable substructures have been investigated mainly for finite directed graphs.
Given a graph $G$, let \emph{$\mathcal{D}(G)$} be the directed graph obtained from $G$ by replacing each edge $uv \in E(G)$ by directed edges $(u,v), (v,u)$.
Seymour and Thomassen~\cite{seymour1987characterization} proved that directed graphs $\mathcal{D}(C)$ for odd cycles $C$ are the unavoidable butterfly minors of finite, even directed graphs.
The unavoidable strongly $2$-connected butterfly minors of finite, strongly $2$-connected directed graphs were characterised by Wiederrecht~\cite{wiederrecht2020digraphs}.
Furthermore, the author showed in the first paper of this series:
\begin{thm}[\cite{finite}*{Theorem 1.1}]
	There exists a function $f$ such that for every $n \in \NN$ and for every set $U$ of vertices of strongly connected directed graph $D$ with $|U| \geq f(n)$ there is a butterfly minor of $D$ that is either shaped by a star or shaped by a comb, and has $n$ teeth in $U$.
\end{thm}
\noindent
This result implies that directed cycles, directed graphs $\mathcal{D}(K_{1,n})$ and directed graphs $\mathcal{D}(P_n)$ for some path $P_n$ are the unavoidable strongly connected butterfly minors of finite, strongly connected directed graphs~\cite{finite}.

For infinite directed graphs only the following result by B\"urger and Melcher~\cite{burger2020ends}*{Lemma 3.2} was known:

\begin{lem}\label{lem:buerger_melcher}
	For every infinite set $U$ of vertices of a strongly connected directed graph $D$, there exists an infinite in-arborescence (out-arborescence) in $D$, whose leaves are in $U$ and whose underlying undirected graph is either a subdivided star or a comb.
\end{lem}
\noindent
\cref{lem:buerger_melcher} is a direct consequence of the star-comb lemma, but unlike the star-comb lemma, it does not preserve the connectivity of its host graph.

In this paper we take on the task of investigating the unavoidable strongly connected substructures of infinite strongly connected directed graphs.
More precisely, we transfer the star-comb lemma to strongly connected directed graphs and obtain a directed variant of~\labelcref{Fact:infinite_undirected_simple} as a consequence.
As in the finite setting, we express the unavoidable substructures in terms of butterfly minors.
The unavoidable strongly connected butterfly minors are more complex than stars and combs from the star-comb lemma since we have to preserve two directed paths, namely a $v$--$w$~path and a $w$--$v$~path, for each pair of vertices $v, w$.

The union of an out-oriented ray $R$ rooted at $r$ and edges $(e_v)_{v \in V(R) \setminus \{r\}}$, where $e_v$ has head $r$ and tail $v$, is called a \emph{dominated directed ray}.
Furthermore, the union of an in-oriented ray $R$ rooted at $r$ and edges $(e_v)_{v \in V(R) \setminus \{r\}}$, where $e_v$ has head $v$ and tail $r$, is also called a \emph{dominated directed ray}.
See \cref{fig:ddray}.

\begin{figure}[ht]
	\begin{tikzpicture}

		\draw[edge] (0.1,1) to (0.9,1);
		\draw[edge] (1.1,1) to (1.9,1);
		\draw[edge] (2.1,1) to (2.9,1);
		\draw[edge] (3.1,1) to (3.9,1);
		
		\draw[edge,path fading=east] (4.1,1) to (5,1);
		\draw[edge,path fading=east] (5,1) to[bend right] (0,1.1);

		\draw[edge] (1,1.1) to[bend right] (0,1.1);
		\draw[edge] (2,1.1) to[bend right] (0,1.1);
		\draw[edge] (3,1.1) to[bend right] (0,1.1);
		\draw[edge] (4,1.1) to[bend right] (0,1.1);
		
		\draw[LavenderMagenta] (0,1) node {\huge.};
		\draw[LavenderMagenta] (1,1) node {\huge.};
		\draw[LavenderMagenta] (2,1) node {\huge.};
		\draw[LavenderMagenta] (3,1) node {\huge.};
		\draw[LavenderMagenta] (4,1) node {\huge.};

		\draw[edge] (6.9,1) to (6.1,1);
		\draw[edge] (7.9,1) to (7.1,1);
		\draw[edge] (8.9,1) to (8.1,1);
		\draw[edge] (9.9,1) to (9.1,1);

		\draw[edge] (6,1.1) to[bend left] (7,1.1);
		\draw[edge] (6,1.1) to[bend left] (8,1.1);
		\draw[edge] (6,1.1) to[bend left] (9,1.1);
		\draw[edge] (6,1.1) to[bend left] (10,1.1);
		
		\draw[LavenderMagenta] (6,1) node {\huge.};
		\draw[LavenderMagenta] (7,1) node {\huge.};
		\draw[LavenderMagenta] (8,1) node {\huge.};
		\draw[LavenderMagenta] (9,1) node {\huge.};
		\draw[LavenderMagenta] (10,1) node {\huge.};
		
		\draw[edge,path fading=east] (11,1) to (10.1,1);
		\draw[edge] (10.3,1) to (10.1,1);
		\draw[edge,path fading=east] (6,1.1) to[bend left] (11,1);

	\end{tikzpicture}
	\caption{Dominated directed rays.}
	\label{fig:ddray}
\end{figure}

We say that a directed graph $D$ is \emph{shaped by a star} (see~\cref{fig:ddray,fig:shaped_by_star}) if either $D$ is a dominated directed ray or there exists a subdivided star $S$ such that
\begin{itemize}
	\item $D= \mathcal{D}(S)$, or
	\item $D$ is obtained from $\mathcal{D}(S)$ by replacing the unique vertex $w \in V(\mathcal{D}(S))$ of infinite degree by a dominated directed ray $R$ with $V(R)=\{w_n: n \in \mathrm{N}_S(w)\}$ such that the edges $(w,n), (n,w)$ become incident with $w_n$ for every $n \in \mathrm{N}_{S}(w)$.
\end{itemize}

\begin{figure}[ht]
	\begin{tikzpicture}
		
		\foreach \a in {1.2,1.9,2.4,3.7}{
			
			\draw[edge, opacity=(1/\a)] ({0.9*cos(360*(1/\a))},{0.9*sin(360*(1/\a))}) to ({0.1*cos(360*(1/\a))},{0.1*sin(360*(1/\a))});
			\draw[edge, opacity=(1/\a)] ({0.1*cos(360*(1/\a))},{0.1*sin(360*(1/\a))}) to ({0.9*cos(360*(1/\a))},{0.9*sin(360*(1/\a))});
			\draw[black, opacity=(1/\a)]  ({1*cos(360*(1/\a))},{1*sin(360*(1/\a))}) node {\huge .};
			\draw[edge, opacity=(1/\a)] ({1.9*cos(360*(1/\a))},{1.9*sin(360*(1/\a))}) to ({1.1*cos(360*(1/\a))},{1.1*sin(360*(1/\a))});
			\draw[edge, opacity=(1/\a)] ({1.1*cos(360*(1/\a))},{1.1*sin(360*(1/\a))}) to ({1.9*cos(360*(1/\a))},{1.9*sin(360*(1/\a))});
			\draw[LavenderMagenta, opacity=(1/\a)]  ({2*cos(360*(1/\a))},{2*sin(360*(1/\a))}) node {\huge .};
		}
		
		\foreach \a in {1,1.5,3,4.5,5.4}{
			
			\draw[edge, opacity=(1/\a)] ({1.9*cos(360*(1/\a))},{1.9*sin(360*(1/\a))}) to ({0.1*cos(360*(1/\a))},{0.1*sin(360*(1/\a))});
			\draw[edge, opacity=(1/\a)] ({0.1*cos(360*(1/\a))},{0.1*sin(360*(1/\a))}) to ({1.9*cos(360*(1/\a))},{1.9*sin(360*(1/\a))});
			\draw[LavenderMagenta, opacity=(1/\a)]  ({2*cos(360*(1/\a))},{2*sin(360*(1/\a))}) node {\huge .};
		}
		
		\draw (0,0) node {\huge.};
		
		\draw[edge] (4.1,0) to (4.9,0);
		\draw[edge] (5.1,0) to (5.9,0);
		\draw[edge] (6.1,0) to (6.9,0);
		\draw[edge] (7.1,0) to (7.9,0);
		
		\draw[edge,path fading=east] (8.1,0) to (9,0);
		\draw[edge,path fading=east] (9,0) to[bend right] (4,0.1);

		\draw[edge] (5,0.1) to[bend right] (4,0.1);
		\draw[edge] (6,0.1) to[bend right] (4,0.1);
		\draw[edge] (7,0.1) to[bend right] (4,0.1);
		\draw[edge] (8,0.1) to[bend right] (4,0.1);
		
		\draw[black] (4,0) node {\huge.};
		\draw[black] (5,0) node {\huge.};
		\draw[black] (6,0) node {\huge.};
		\draw[black] (7,0) node {\huge.};
		\draw[black] (8,0) node {\huge.};
		
		\draw[LavenderMagenta] (4,-2) node {\huge.};
		\draw[black] (5,-1) node {\huge.};
		\draw[LavenderMagenta] (5,-2) node {\huge.};
		\draw[LavenderMagenta] (7,-2) node {\huge.};
		\draw[black] (6,-1) node {\huge.};
		\draw[LavenderMagenta] (6,-2) node {\huge.};
		\draw[LavenderMagenta] (8,-2) node {\huge.};
		
		\draw[edge] (4,-0.1) to (4,-1.9);
		\draw[edge] (4,-1.9) to (4,-0.1);
		
		\draw[edge] (5,-0.1) to (5,-0.9);
		\draw[edge] (5,-0.9) to (5,-0.1);
		\draw[edge] (5,-1.1) to (5,-1.9);
		\draw[edge] (5,-1.9) to (5,-1.1);
		
		\draw[edge] (6,-0.1) to (6,-0.9);
		\draw[edge] (6,-0.9) to (6,-0.1);
		\draw[edge] (6,-1.1) to (6,-1.9);
		\draw[edge] (6,-1.9) to (6,-1.1);
		
		\draw[edge] (7,-0.1) to (7,-1.9);
		\draw[edge] (7,-1.9) to (7,-0.1);
		
		\draw[edge] (8,-0.1) to (8,-1.9);
		\draw[edge] (8,-1.9) to (8,-0.1);
		
	\end{tikzpicture}
	\caption{Two directed graphs shaped by a star.}
	\label{fig:shaped_by_star}
\end{figure}

\noindent
A directed graph $D$ is \emph{shaped by a comb} (see~\cref{fig:shaped_by_comb}) if there exists a comb $C$ such that either
\begin{itemize}
	\item $D= \mathcal{D}(C)$, or
    \item $D$ is obtained from $\mathcal{D}(C)$ by replacing each junction $j$ of $C$ by a directed cycle of length $3$ such that, if $d_C(j)=3$, the three strong components of $\mathcal{D}(C)-j$ become incident with distinct vertices of this directed cycle and otherwise the tooth $j$ is a vertex of this directed cycle and the two strong components of $\mathcal{D}(C)-j$ become incident with distinct other vertices of this directed cycle.
\end{itemize}
\begin{figure}[ht]
	\begin{tikzpicture}
		
		\draw[edge][] (0.25,0) to (0.75,0);
		\draw[edge] ({0.75},0) to ({0.25},0);
		\draw[edge] ({0.2},0.05) to ({0.2},0.95);
		\draw[edge] ({0.2},0.95) to ({0.2},0.05);
		\draw[] ({0.2},0) node {.};
		
		\foreach \a in {1, 2, 4}{

			\draw[LavenderMagenta] (\a,1) node {\huge.};

			\draw[edge] ({\a-0.15},0) to ({\a+0.15},0);
			\draw[edge] ({\a+0.15},0.05) to ({\a+0.05},0.35);
			\draw[edge] ({\a-0.05},0.35) to ({\a-0.15},0.05);

			\draw[edge] ({\a},0.45) to ({\a},0.95);
			\draw[edge] ({\a},0.95) to ({\a},0.45);
			
			\draw[] ({\a-0.2},0) node {.};
			\draw[] ({\a+0.2},0) node {.};
			\draw[] ({\a},0.4) node {.};	
		}
		
		\foreach \a in {1,2, 3}{
			
			\draw[edge] ({\a+0.25},0) to ({\a+0.75},0);
			\draw[edge] ({\a+0.75},0) to ({\a+0.25},0);
		}
	
		\draw[] (2.8,0) node {.};
		\draw[] (3.2,0) node {.};
		\draw[LavenderMagenta] (3,0.4) node {\huge.};
		\draw[edge] (3.15,0.05) to (3.05,0.35);
		\draw[edge] (2.95,0.35) to (2.85,0.05);
		\draw[edge] (2.85,0) to (3.15,0);
		
		\draw[edge] ({4.35},0) to ({4.25},0);
		\draw[edge,path fading=east] ({4.75},0) to ({4.25},0);
		
		
		\foreach \a in {-6, -2}{
			\draw[LavenderMagenta] (\a,1) node {\huge.};
			\draw[] (\a,0) node {.};	
			
			\draw[edge] ({\a},0.05) to ({\a},0.95);
			\draw[edge] ({\a},0.95) to ({\a},0.05);
			
		}
	
		\draw[LavenderMagenta] (-3,0) node {\huge.};
		
		\foreach \a in {-5,-4}{
			\draw[LavenderMagenta] (\a,1) node {\huge.};
			\draw[black] (\a,0.5) node {.};
			\draw[] (\a,0) node {.};	
			
			\draw[edge] ({\a},0.05) to ({\a},0.45);
			\draw[edge] ({\a},0.45) to ({\a},0.05);
			\draw[edge] ({\a},0.55) to ({\a},0.95);
			\draw[edge] ({\a},0.95) to ({\a},0.55);
			
		}
		
		\foreach \a in {-6, ..., -3}{
			
			\draw[edge] ({\a+0.05},0) to ({\a+0.95},0);
			\draw[edge] ({\a+0.95},0) to ({\a+0.05},0);
		}
		
		\draw[LavenderMagenta] (0.2,1) node {\huge.};
		
		\draw[edge] ({-1.85},0) to ({-1.95},0);
		\draw[edge,path fading=east] ({-1.5},0) to ({-1.95},0);

	\end{tikzpicture}
	\caption{Two directed graphs shaped by a comb.}
	\label{fig:shaped_by_comb}
\end{figure}
\noindent
Let $T$ be either a subdivided star or a comb. The \emph{teeth} of a directed graph shaped by $T$ are the vertices that correspond to teeth of $T$.
Furthermore, we call all vertices of a dominated directed ray \emph{teeth}.

A \emph{chain of triangles} (see~\cref{fig:chain_of_triangles}) is the union of the directed graph with vertex set $\bigcup_{i \in \NN} \{a_i,b_i,c_i\}$ and edge set 
\begin{align*}
	\{(a_i, b_i): i \in \NN \} \cup \{(b_i, c_i): i \in \NN \} \cup \{(c_i, a_i): i \in \NN \} \\ \cup \{(a_{i+1},a_i): i \in \NN \} \cup \{(c_{i},c_{i+1}): i \in \NN \}
\end{align*}
and disjoint \emph{double paths} $b_iF_iy_i$, that is directed graphs $\mathcal{D}(P)$ for a path $P$, for every $i \in \NN$ each intersecting $\bigcup_{i \in \NN} \{a_i,b_i,c_i\}$ only in $b_i$.
We call the vertices $(y_i: i \in \NN)$ the \emph{teeth} of this chain of triangles.

\begin{figure}[ht]
	\begin{tikzpicture}
		
		\draw[edge] (1.1,2) to (1.9,2);
		\draw[edge] (1.9,2) to (1.1,2);
		\draw[LavenderMagenta] (2,2) node {\huge.};

		\draw[edge] (9.1,2) to (9.9,2);
		\draw[edge] (9.9,2) to (9.1,2);
		\draw[LavenderMagenta] (10,2) node {\huge.};
		\draw[edge] (10.1,2) to (10.9,2);
		\draw[edge] (10.9,2) to (10.1,2);
		\draw[LavenderMagenta] (11,2) node {\huge.};
		
		\foreach \a in {0,4,8}{
			\draw[edge] ({\a+0.1},1.05) to ({\a+0.9},1.9);
			\draw[edge] (\a+0.9,2.1) to ({\a+0.1},2.9);
			\draw[edge] (\a,2.9) to ({\a},1.1);
			
			\draw[LavenderMagenta] (\a,1) node {\huge.};
			\draw[LavenderMagenta] ({\a+1},2) node {\huge.};
			\draw[LavenderMagenta] (\a,3) node {\huge.};
		}
		
		\foreach \a in {0,4}{
			\draw[edge] ({\a+0.1},3) to ({\a+3.9},3);
			\draw[edge] ({\a+3.9},1) to ({\a+0.1},1);
		}
		
		\draw[edge, path fading=east] ({8.1},3) to ({11.9},3);
		\draw[edge, path fading=east] ({11.9},1) to ({8.1},1);
		\draw[edge] ({8.2},1) to ({8.1},1);
		
		\foreach \a in {1,2,3}{
			\draw ({4*\a-4},0.7) node {$a_\a$};
			\draw ({4*\a-3.4},2) node {$b_\a$};
			\draw ({4*\a-4},3.3) node {$c_\a$};
		}
		\draw (2.4,2) node {$y_1$};
		\draw (5.4,2) node {$y_2$};
		\draw (11.4,2) node {$y_3$};

	\end{tikzpicture}
	\caption{A chain of triangles.}
	\label{fig:chain_of_triangles}
\end{figure}

Note that directed graphs shaped by a star or shaped by a comb and chains of triangles are strongly connected.
Our main result reads as follows:
\begin{restatable}{thm}{infinite}\label{main_theorem}
	For every infinite set $U$ of vertices of a strongly connected directed graph $D$ there exists a butterfly minor of $D$ that is either shaped by a star, shaped by a comb or is a chain of triangles, and has all teeth in $U$.
\end{restatable}
\noindent
We derive a variant of~\labelcref{Fact:infinite_undirected_simple} for directed graphs from \cref{main_theorem}:
\begin{cor}
	Every infinite strongly connected directed graph contains $\mathcal{D}(K_{1, \infty}), \mathcal{D}(R)$ for an undirected ray $R$ or a dominated directed ray as a butterfly minor.
\end{cor}

We show in~\cref{sec:necessity} that all six types of directed graphs in the statement of~\cref{main_theorem} are indeed necessary. Note that if we allow strong minors instead of butterfly minors in~\cref{main_theorem} only directed graphs $\mathcal{D}(K_{1, \infty})$, $\mathcal{D}(R)$ for an undirected ray $R$ and dominated directed rays are necessary, since all other types can be strongly contracted to one of these types.

This paper is organised as follows:
We introduce notations and basic properties of directed graphs in~\cref{sec:preliminaries}.
In~\cref{sec:centre_infinite} we prove the existence of certain strongly connected subgraphs which is the basis for the proof of~\cref{main_theorem}.
In~\cref{sec:forcing_stars} we provide sufficient conditions for the existence of butterfly minors shaped by a star and prove~\cref{main_theorem} in~\cref{sec:proof}.
Finally, we argue in~\cref{sec:necessity} that all six types of butterfly minors in~\cref{main_theorem} are indeed necessary.

\section{Preliminaries}\label{sec:preliminaries}
For standard notations we refer to Diestel's book~\cite{diestel}. We set $[\ell]:= \{1, \dots, \ell\}$ for $\ell \in \NN$.
A \emph{directed $x_1$--$x_\ell$~walk} of a directed graph is an alternating sequence $W=x_1e_2x_2 \dots x_{\ell -1} e_{\ell - 1} x_\ell$ of vertices $(x_i: i \in [\ell])$ and edges $(e_i: i \in [\ell - 1])$ such that $x_i$ is the tail of $e_i$ and $x_{i+1}$ is the head of $e_i$.
We call also the directed graph induced by $W$ \emph{directed walk}.
We say that $x_1$ is the \emph{startvertex} and $x_\ell$ the \emph{endvertex} of $W$.
Further, we call the vertices $x_2, \dots, x_{\ell - 1}$ the \emph{internal vertices} of the directed walk $W$.
The directed walk $W$ is a \emph{directed $X$--$Y$~walk} for two given sets $X, Y$ of vertices if $x_1 \in X$, $x_\ell \in Y$ and the internal vertices of $W$ are not contained in $X \cup Y$.
Furthermore, a \emph{directed $A$--$B$~walk} for subgraphs $A, B$ is a $V(A)$--$V(B)$~walk.
If all vertices $(x_i: i \in [\ell])$ are disjoint and all edges $(e_i: i \in [\ell - 1])$ are disjoint, we call $W$ a \emph{directed path}.

Next, we define an order on the vertex set of a directed path $P$ as follows: given $v, w \in V(P)$, we say $v \leq_P w$ if $w\in V(vP)$.
We define $Px$ to be the initial segment of a directed path $P$ ending in a vertex $x \in V(P)$ and define $xP$ analogously.
Furthermore, $P x Q$ refers to the directed path obtained by the concatenation of $Px$ and $xQ$ along $x$, given some directed path $Q$ with $x \in V(P) \cap V(Q)$.
A directed graph $D$ is called \emph{strongly connected}, if for every $u, v \in V(D)$ there exists a directed $u$--$v$~path in $D$.

A \emph{double path} is the directed graph $\mathcal{D}(P)$ obtained for an undirected path $P$.
A \emph{directed cycle} is a sequence $C = x_1e_2x_2 \dots x_{\ell -1} e_{\ell - 1} x_\ell$ of distinct vertices $(x_i: i \in [\ell - 1])$ with $x_1 = x_\ell$ and distinct edges $(e_i: i \in [\ell - 1])$ such that $x_i$ is the tail of $e_i$ and $x_{i+1}$ is the head of $e_i$.

In this paper we make repeatedly implicit use of Ramsey's theorem:
\begin{thm}{\cite{ramsey}} \label{ramsey2}
	Every complete infinite graph with a $k$-colouring of the edges contains an infinite monochromatic clique.
\end{thm}

\subsection{Arborescences}
We call a directed graph an \emph{in-arborescence}, if its underlying undirected graph is a rooted tree and all edges are directed towards the root.
The terms root and leaves transfer from the underlying undirected graph to the in-arborescence.
An \emph{in-ray} is an in-arborescences whose underlying undirected graph is a ray with the unique vertex of degree $1$ being its root.
An \emph{in-star} is an in-arborescence whose underlying undirected graph is a star with the centre being its root.
\emph{Out-arborescences}, \emph{out-rays}, \emph{out-stars} are defined analogously.
Given an in-ray or an out-ray $R$ rooted at $r$, we set, given $v, w \in V(R)$, $v \leq_R w$ if $v \in V(rRw)$.

\begin{prop}\label{prop:spanning_arborescences}
	Let $U$ be a countable set of vertices of a strongly connected directed graph~$D$.
	For every $r \in V(D)$ there exist an in-arborescence and an out-arborescence rooted at $r$ that contain $U$ and whose leaves are contained in $U$.
\end{prop}
\begin{proof}
	Let $r \in V(D)$ be arbitrary and let $\{u_n: n \in \NN\}$ be some enumeration of $U$.
	We construct recursively an increasing sequence $(T_n)_{n \in \NN}$ of in-arborescences rooted at $u$ such that $T_n$ contains $\{u_i: i \in [n]\}$ and has all its leaves in $U$.
	Set $T_0:= \{u\}$.
	If $T_{n -1 }$ has been defined for some $n \in \NN$, set $T_n:= T_{n - 1}$ if $u_n \in T_{n - 1}$.
	Otherwise pick some directed $u_n$--$T_{n-1}$~path $P_n$, which exists since $D$ is strongly connected, and set $T_n:= T_{n - 1} \cup P_n$.
	Then the in-arborescence $\bigcup_{n \in \NN} T_n$ is as desired.
	By reversing all orientations in the previous arguments we obtain the desired out-arborescences.
\end{proof}

We state the following strengthening of B\"urger and Melcher's result,~\cref{lem:buerger_melcher}:

\begin{prop}\label{prop:star_comb_arborescences}
	Let $U$ be some infinite set of vertices of an in-arborescence (out-arborescence) $T$ rooted at $r$.
	Then $T$ contains either an in-ray (out-ray) $R$ rooted at $r$ together with infinitely many (possibly trivial) disjoint directed $U$--$R$~paths ($R$--$U$~paths) or a subdivided in-star (out-star) with all leaves in $U$.
\end{prop}
\noindent
B\"urger and Melcher's proof of~\cref{lem:buerger_melcher}~\cite{burger2020ends} also shows~\cref{prop:star_comb_arborescences}.
For the convenience of the reader, we present a proof nevertheless:
\begin{proof}
	We apply the star-comb lemma to the underlying undirected graph of $T$ and the set $U\setminus \{r\}$ to obtain either a subdivided star $S$ or a comb $C$ with all leaves in $U$.
	In the former case, the induced subgraph of $S$ contains an in-star with all leaves in $U$.
	In the latter case, there is a rooted in-ray $R$ in $T$ that has a tail in the back of the comb $C$.
	Then almost all teeth of $C$ induce the desired directed $U$--$R$~paths.
		By reversing all orientations, we obtain the desired out-ray/out-star.
\end{proof}

\subsection{Laced paths} \label{sec:basic}
In this subsection, we state properties about pairs of directed paths that intersect in a simple way.
We call two directed paths $P$ and $Q$ \emph{laced} if either $P$ and $Q$ are disjoint or there exist vertices $x_1, y_2, \dots, x_\ell, y_\ell \in V(P) \cap V(Q)$ and $\ell \in \NN$ such that
\begin{itemize}
	\item $x_1 \leq_P y_1 <_P \dots <_P x_{\ell} \leq_P y_{\ell}$,
	\item $x_\ell \leq_Q y_\ell <_Q \dots <_Q x_1 \leq_Q y_1$,
	\item $x_i P y_i = x_i Q y_i$ for every $i \in [\ell]$,
	\item the segments $y_\ell Q x_{\ell - 1}, \dots, y_2 Q x_{ 1}$ are internally disjoint to $P$, and
	\item $Q x_\ell$ intersects $P$ only in $x_\ell$ and $y_1 Q$ intersects $P$ only in $y_1$.
\end{itemize}
See \cref{fig:laced_paths} for an illustration.

\begin{figure}[ht]
	\begin{tikzpicture}
		
		\draw (3,1.6) node {$x_1$};
		\draw (4,0.4) node {$y_1$};
		\draw (5,0.4) node {$x_2$};
		\draw (6,1.6) node {$y_2$};
		\draw (7,1.6) node {$x_3$};
		\draw (8,0.4) node {$y_3$};
		
		\foreach \a in {1,...,9}{
			\draw[edge][PastelOrange] ({\a+0.1},1) to ({\a+0.9},1);
			\draw[LavenderMagenta] (\a,1) node {\huge .};
			\draw[LavenderMagenta] ({\a+1},1) node {\huge .};
		}
		\draw[edge, dashed][CornflowerBlue] (3.1,1) to (3.9,1);
		\draw[edge, dashed][CornflowerBlue] (5.1,1) to (5.9,1);
		\draw[edge, dashed][CornflowerBlue] (7.1,1) to (7.9,1);
		
		\draw[CornflowerBlue] (4.5,2) node {$Q$};
		
		\draw[PastelOrange] (2,2) node {$P$};
		
		\draw[edge][CornflowerBlue] (4,0.9) to[bend left] (2.1,0);
		\draw[edge][CornflowerBlue] (8,0.9) to[bend left] (5,0.9);
		\draw[edge][CornflowerBlue] (6,1.1) to[bend right] (3,1.1);
		\draw[edge][CornflowerBlue] (8.9,2) to[bend right] (7,1.1);
		
		\draw[LavenderMagenta] (2,0) node {\huge .};
		\draw[LavenderMagenta] (9,2) node {\huge .};
		
	\end{tikzpicture}
	\caption{Two laced directed paths.}
	\label{fig:laced_paths}
\end{figure}

\begin{prop}[\cite{hatzel2024generating}*{Lemma 4.3}]\label{prop:laced_paths}
	Let $D$ be a directed graph and let $P$ be some directed path in~$D$.
	Further, let $Q$ be a directed $a$--$b$~path in $D$ for some vertices $a, b \in V(D)$.
	Then there exists a directed $a$--$b$~path $Q'$ in $D$ with $Q' \subseteq P \cup Q$ such that $P$ and $Q'$ are laced.
\end{prop}

\begin{obs}\label{obs:double_path}
	Let $D$ be a directed graph and let $a,b \in V(D)$.
	Further, let $P$ be a directed $a$--$b$~path and let $Q$ be a directed $b$--$a$~path such that $P$ and $Q$ are laced.
	Then the union $P \cup Q$ can be butterfly contracted to a double path starting in $a$ and ending in $b$.
\end{obs}

\subsection{Knit rays}
In this subsection, we extend the concept of laced paths to pairs of in-rays and out-rays that have infinitely many vertices in common.
We call a pair of an out-ray $R$ and an in-ray $S$ with common root $r$ \emph{knit} if there exist vertices
$x_1, y_1, x_2, y_2, \dots \in V(R) \cap V(S)$ such that
\begin{itemize}
	\item $r <_R x_1 \leq_R y_1 <_R x_{2} \leq_R y_{2} <_R \dots$,
	\item $r <_S y_1 \leq_S x_1 < y_2 \leq_S x_2 <_S \dots$,
	\item $x_i R y_i = y_i S x_i$ for every $i \in \NN$,
	\item the segment $x_i S y_{i+1}$ is internally disjoint to $R$ for every $i \in \NN$, and
	\item the segment $rSy_1$ is internally disjoint to $R$.
\end{itemize}
See \cref{fig:knit_rays} for an illustration.

\begin{figure}[ht]
	\begin{tikzpicture}

		\foreach \a in {2,...,7}{
			\draw[edge][PastelOrange] ({\a+0.1},1) to ({\a+0.9},1);
			\draw[LavenderMagenta] (\a,1) node {\huge .};
			\draw[LavenderMagenta] ({\a+1},1) node {\huge .};
		}
		
		\draw[edge][PastelOrange, path fading=east] ({8.1},1) to ({9.9},1);
		
		\draw[edge, CornflowerBlue] (4,0.9) to[bend left] (2,0.9);
		\draw[edge, CornflowerBlue] (8,0.9) to[bend left] (5,0.9);
		\draw[edge, CornflowerBlue] (6,1.1) to[bend right] (3,1.1);
		\draw[edge, CornflowerBlue, path fading=east] (10,1.1) to[bend right] (7,1.1);
		
		\draw (2,0.5) node {$r$};
		\draw (3,1.5) node {$x_1$};
		\draw (4,0.5) node {$y_1$};
		\draw (5,0.5) node {$x_2$};
		\draw (6,1.5) node {$y_2$};
		\draw (7,1.5) node {$x_3$};
		\draw (8,0.5) node {$y_3$};
		
		\draw[PastelOrange] (2,2) node {$R$};
		
		\draw[CornflowerBlue] (5,2) node {$S$};
		
		\draw[edge, CornflowerBlue, dashed] (3.1,1) to (3.9,1);
		\draw[edge, CornflowerBlue, dashed] (5.1,1) to (5.9,1);
		\draw[edge, CornflowerBlue, dashed] (7.1,1) to (7.9,1);

	\end{tikzpicture}
	\caption{Two knit rays.}
	\label{fig:knit_rays}
\end{figure}

We remark that, given two knit rays $R$ and $S$, there exists a sequence of directed cycles $(C_i)_{i \in \NN}$ with the property that for every $j \neq k \in \NN$ $C_j\cap C_k$ is a directed path if $|j-k|=1$ and empty otherwise, and $\bigcup_{i \in \NN} C_i = R \cup S$.

\begin{prop}\label{prop:in_tree_or_knit_ray}
	Let $D$ be a strongly connected directed graph with an out-ray $R$.
	Then there exist $r \in V(R)$ and either
	\begin{itemize}
		\item an in-arborescence $T$ rooted at $r$ that intersects $V(R)$ exactly in $r$ and in infinitely many leaves of $T$, or
		\item an in-ray $S$ rooted at $r$ that is knit with $rR$.
	\end{itemize}
\end{prop}

\begin{proof}
		If there exists a vertex $r \in V(R)$ such that there are infinitely many directed $R$--$r$~paths that start in distinct vertices of $R$, then there exists an in-arborescence $T$ rooted at $r$ that intersects $V(R)$ exactly in $r$ and in infinitely many leaves of $T$ by~\cref{prop:spanning_arborescences}.
	
	Thus we can assume that for every $v \in V(R)$ there exists a $\leq_R$-maximal vertex $\alpha(v)$ in which a directed $R$--$v$~path starts.
	Let $r$ be the root of $R$.
	We define a $\leq_R$-increasing sequence $(y_n)_{n \in \NN_0}$ of vertices of $R$ with $y_0:= r$ and such that $y_{n+1}$ is $\leq_R$-maximal along $\{\alpha(v): v \leq_R y_n \}$ for every $n \in \NN_0$.
	Since $D$ is strongly connected, there exists some directed $R$--$R$~path starting in $\{v \in V(R): v >_R y_n \}$ and ending in $\{v \in V(R): v \leq_R y_n \}$, which implies that $(y_n)_{n \in \NN_0}$ is strictly $\leq_R$-increasing.
	
	Let $Q_n$ be some directed $R$--$R$~path starting in $y_{n}$ and ending in $\{v \in V(R): v \leq_R y_{n-1} \}$ for every $n \in \NN$, which exists by the choice of $y_{n}$.
	The endpoint $x_{n-1}$ of $Q_n$ holds $y_{n-2} < x_{n-1} \leq y_{n-1}$, by the choice of $y_{n-1}$.
	Furthermore, note that the directed paths $(Q_n)_{n \in \NN}$ are internally vertex-disjoint by the choice of $(y_n)_{n \in \NN_0}$.
	Thus the concatenation $r Q_1 y_1 R x_1 Q_2 y_2 R x_2 Q_3 y_3 \dots$ is an in-ray that is knit with $rR$, as desired.
\end{proof}

\subsection{Butterfly minors}
Before we introduce butterfly minors for infinite directed graphs, we recall the definition of butterfly minors for finite directed graphs:
An edge $e$ is \emph{butterfly contractable} if either $e$ is the only outgoing edge of its tail or $e$ is the only ingoing edge of its head.
A finite directed graph $H$ is a \emph{butterfly minor} of a finite directed graph $D$ if $H$ can be obtained from a subgraph of $D$ by repeatedly contracting butterfly contractable edges.
Amiri, Kawarabayashi, Kreutzer and Wollan presented \cite{amiri2016erdos}*{Lemma 3.2} an equivalent definition of butterfly minors for finite directed graphs via tree-like models.

We transfer their definition verbatim to infinite directed graphs:
Given directed graphs $D$ and $H$, a map $\mu $ assigning every $e\in E(H)$ an edge of $D$ and every $v \in V(H)$ a subgraph of $D$, such that
\begin{itemize}
	\item $\mu (v) \cap \mu (w) = \emptyset$ for every $v \neq w \in V(H)$,
	\item for every $v \in V(H)$ the directed graph $\mu(v)$ is the union of an in-arborescence $T_{\mathrm{in}}^{\mu(v)}$ and an out-arborescence $T_{\mathrm{out}}^{\mu(v)}$ which only have their roots in common, and
	\item for every $e=(u,v) \in E(H)$ the edge $\mu(e)$ has its tail in $T_{\mathrm{out}}^{\mu(u)}$ and its head in $T_{\mathrm{in}}^{\mu(v)}$
\end{itemize}
is called a \emph{tree-like model of $H$ in $D$}.
A directed graph $H$ is a \emph{butterfly minor} of a directed graph~$D$ if there exists a tree-like model of $H$ in $D$.

Alternatively, one could define butterfly minors of infinite directed graphs via sequences of butterfly contractions with a suitable definition for limit steps.
In the next paragraph, we show that this generalises our notion of butterfly minors in the following way:
$\mu(v)$ can also be a directed graphs whose vertices have out-degree precisely one and no edge $\mu(e)$ has tail in $\mu(v)$, and, similarly, $\mu(v)$ can be a directed graph whose vertices have in-degree precisely one and no edge $\mu(e)$ has head in $\mu(v)$.
However, in these cases, there is in general no vertex in $\mu(v)$ that can reach all edges leaving $\mu(v)$ and can be reach from edges entering $\mu(v)$:
For example, $\mu(v)$ can be an out-ray $R$ together with infinitely many edges that enter $\mu(v)$ in distinct vertices of $R$.
As the existence of a vertex in $\mu(v)$ representing the connectivity of $v$ (see \cref{propconnectedness}) is fundamental in the context of finite directed graphs, we define butterfly minors in the more restrictive way.

Now, we show that sequences of butterfly contractions generalise our notion of butterfly minors in the described way:
Let $C$ be a weakly connected component of edges that will be contracted by such a sequence.
We assume that for every $e \in E(C)$, there are edges leaving or entering both weak components of $C - e$, since otherwise $e$ can be deleted instead.
Note that the underlying undirected graph of $C$ is a tree.
As in the finite case,
for every subgraph $P$ of $C$ whose underlying undirected graph is a path with endvertices $u$ and $v$ holds:
\begin{itemize}
	\item if there is an edge entering $C$ in $u$ and an edge leaving $C$ in $v$, then $P$ is a directed $u$--$v$~path,
	\item if there are edges entering $C$ in $u$ and $v$, then every internal vertex of $P$ has out-degree at most $1$ in $P$, and
	\item if there are edges leaving $C$ in $u$ and $v$, then every internal vertex of $P$ has in-degree at most $1$ in $P$.
\end{itemize}
This implies that, if there exist edges leaving and entering $C$, then $C$ is a union of an in-arborescence and an out-arborescence as in the tree-like model.
If there are no edges leaving $C$, then all vertices of $C$ have out-degree at most $1$ in $C$.
Similarly, if there are no edges entering $C$, then all vertices of $C$ have in-degree at most $1$ in $C$.
Thus $C$ has the desired shape, and therefore the sequence of butterfly contractions can be displayed in the described way.
Conversely, it is natural that every such directed graph can be contracted to a single vertex by some sequence of butterfly contractions.

We turn our attention to properties of butterfly minors.
We remark that a directed graph obtained from a subgraph of directed graph $D$ by a finite sequence of butterfly contractions is a butterfly minor of $D$.

\begin{obs}
	Let $D_1$, $D_2$, $D_3$ be directed graphs such that $D_1$ is a butterfly minor of $D_2$ and $D_2$ is a butterfly minor of $D_3$. Then $D_1$ is a butterfly-minor of $D_3$.
\end{obs}
Let $\Dv$ be the directed graph obtained from a directed graph $D$ by reversing the orientation of every edge in $D$.
\begin{obs}\label{obs:reversing}
	If $H$ is a butterfly minor of $D$, then $\Hv$ is a butterfly minor of $\Dv$.
\end{obs}

An edge $e=(u,v)\in E(D)$ of a directed graph $D$ is called \emph{butterfly contractible} if $e$ is the only outgoing edge of $u$ or the only incoming edge of $v$.
\begin{obs}
	Let $H$ and $D$ be directed graphs. If $H$ is obtained from a subgraph of $D$ by a finite sequence of contractions of butterfly contractible edges, then $H$ is a butterfly minor of $D$.
\end{obs}

The vertex set of a butterfly minor is in general not contained in the vertex set of its host graph.
Therefore in the statement of~\cref{main_theorem} it is not immediately clear how a butterfly minor can have teeth in $U$.
To overcome this problem, we identify, given a tree-like model $\mu$ of $H$ in $D$ and arborescences $(T_{\mathrm{in}}^{\mu(v)})_{v \in V(H)}, (T_{\mathrm{out}}^{\mu(v)})_{v \in V(H)}$, each vertex $u \in V(H)$ with the common root of $T_{\mathrm{in}}^{\mu(u)}$ and $T_{\mathrm{out}}^{\mu(u)}$.
Thus some vertex $u \in V(H)$ is contained in a set $U \subseteq V(D)$ if the common root of $T_{\mathrm{in}}^{\mu(u)}$ and $T_{\mathrm{out}}^{\mu(u)}$ is contained in $U$.

In this paper we will not state tree-like models explicitly.
Instead, we will indicate either the non-trivial in-arborescences and out-arborescences or the butterfly contractible edge that we pick for the construction of a certain butterfly minor.
In the latter case, there might be two choice for a tree-like model, but we will implicitly choose the tree-like model in a way that the vertices of the given set $U$ are preserved.

We show that taking butterfly minors does not create new directed paths:
\begin{prop}\label{propconnectedness}
	Let $H$ be a butterfly minor of a directed graph $D$ and let $v, w \in V(H)$. If there exists a directed $v$--$w$~path in $H$, then there exists a directed $v$--$w$~path in $D$.
\end{prop}
\begin{proof}
	Let $\mu$ be a tree-like model of $H$ in $D$, and let $(T_{\mathrm{in}}^{\mu(v)})_{v \in V(H)}$ and $(T_{\mathrm{out}}^{\mu(v)})_{v \in V(H)}$ be the arborescences of $\mu$.
	Let $P$ be a directed $v$--$w$~path in $H$. For every edge $e =(x,y)$ in $P$ there exists a directed path in $T_{\mathrm{out}}^{\mu(x)}$ that starts in the root of $T_{\mathrm{out}}^{\mu(x)}$ and ends in the tail of $\mu(e)$. Similarly, there exists a directed path in $T_{\mathrm{in}}^{\mu(y)}$ that starts in the head of $\mu(e)$ and ends in the root of $T_{\mathrm{in}}^{\mu(y)}$.
	By concatenation of all these paths and the edge $\mu(e)$ for every $e \in E(P)$ we obtain the desired directed $v$--$w$~path in $D$.
\end{proof}

Finally, we state the unavoidable strongly connected butterfly minors that preserve three given vertices.

\begin{prop}[\cite{finite}*{Proposition 5.2}]\label{prop3star}
	Let $D$ be a strongly connected directed graph and let $u_1, u_2, u_3$ be distinct vertices of $D$.
	Then there exists a butterfly minor of $D$ that is
	\begin{itemize}
		\item a double path containing $u_1, u_2, u_3$,
		\item the union of non-trivial double paths $F_1, F_2, F_3$ that intersect only in a common endpoint and whose other endpoints are $u_1, u_2, u_3$, or
		\item the union of a directed cycle $C$ of length $3$ and disjoint (possibly trivial) double paths $F_1, F_2, F_3$ having an endpoint in $V(C)$ and whose other endpoints are $u_1, u_2, u_3$.
	\end{itemize}
\end{prop}

\section{Constructing the centre}\label{sec:centre_infinite}

In this section we establish the foundation for the proof of~\cref{main_theorem}.
Given an infinite set of vertices $U$, we prove the existence of a certain strongly connected subgraph $A$ together with an infinite family of directed $A$--$U$~paths that intersect only in $A$.
\begin{lem}\label{lemgroundwork2}
	Let $D$ be a strongly connected directed graph. 
	For every infinite set $U \subseteq V(D)$ there exists a subgraph $A$ in $D$
	that is
	\begin{enumerate}[label=(\arabic*)]
		\item\label{itm:centre_1} a single vertex,
		\item\label{itm:centre_2} the union of an out-ray $R_1$, an in-ray $R_2$ and a family $(c_n P_n d_n: n\in \NN)$ of disjoint directed $R_1$--$R_2$~paths such that $R_1$ and $R_2$ intersect only in their common root $u$ and such that $c_n <_{R_1} c_m$, $d_n <_{R_2} d_m$ for every $n < m \in \NN$,
		\item\label{itm:centre_3} the union of an out-ray $R$ with root $r$ and a family $(P_n: n \in \NN)$ of directed $R$--$r$~paths that intersect only in $r$, or
		\item\label{itm:centre_4} the union of directed cycles $(C_j)_{j \in \NN}$ with the property that for every $j \neq k \in \NN$ the intersection $C_j\cap C_k$ is a directed path if $|j-k|=1$ and empty otherwise,
	\end{enumerate}
	such that there exist infinitely many directed $A$--$U$~paths in $D$ that are disjoint if $A$ is of type~\labelcref{itm:centre_2,itm:centre_3,itm:centre_4}, and non-trivial and disjoint outside $A$, if $A$ is of type~\labelcref{itm:centre_1}.
\end{lem}
\noindent
See~\cref{fig:groundwork} for subgraphs of type~\labelcref{itm:centre_2,itm:centre_3,itm:centre_4}.
Note that all types of subgraphs in~\cref{lemgroundwork2} are strongly connected.

\begin{figure}[ht]
	
	\centering
	\begin{subfigure}[b]{0.8\textwidth}
	\begin{tikzpicture}
		
		\draw[edge][PastelOrange] (-0.95,2) to (-0.05,2.95);
		\draw[edge][CornflowerBlue] (0,1) to (-0.95,1.95);
		\draw[LavenderMagenta] (-1,2) node {\huge.};
		\draw (-1,1.5) node {$u$};
		
		\foreach \a in {0,4, 6}{
			\draw[edge] ({\a},2.95) to ({\a},1.05); 
			\draw[edge] ({\a+2},2.95) to ({\a+2},1.05);
			\draw[edge] ({\a+4},2.95) to ({\a+4},1.05); 
			
			\draw[edge][PastelOrange] ({\a+0.1},3) to ({\a+1.9},3);
			\draw[edge][PastelOrange] ({\a+2.1},3) to ({\a+3.9},3);
			\draw[edge][CornflowerBlue] ({\a+1.9},1) to ({\a+0.1},1);
			\draw[edge][CornflowerBlue] ({\a+3.9},1) to ({\a+2.1},1);
		}
		\foreach \a in {1,2,3,4, 5,6}{
			\draw[LavenderMagenta] ({\a*2-2},1) node {\huge.};
			\draw ({\a*2-2},0.5) node {$d_\a$};
			\draw ({\a*2-2},3.5) node {$c_\a$};
			\draw[LavenderMagenta] ({\a*2-2},3) node {\huge.};
		}
		
		\draw[edge][white] ({10.1},3) to ({11.9},3);
		\draw[edge][PastelOrange, path fading=east] ({10.1},3) to ({11.9},3);
		\draw[edge][white] ({11.9},1) to ({10.1},1);
		\draw[edge][CornflowerBlue, path fading=east] ({11.9},1) to ({10.1},1);
		\draw[edge][CornflowerBlue] ({10.3},1) to ({10.1},1);
		
		\draw[CornflowerBlue] (12.5,1) node {$R_2$};
		\draw[PastelOrange] (12.5,3) node {$R_1$};
		
	\end{tikzpicture}
		\caption{A subgraph of type~\labelcref{itm:centre_2}.}
	\end{subfigure} \hfill

	\begin{subfigure}[b]{0.8\textwidth}
	\begin{tikzpicture}

		\draw[edge][PastelOrange] (-1,1.55) to (-0.05,2.95);
		
		\draw[LavenderMagenta] (-1,1.5) node {\huge.};
		\draw (-1,1) node {$r$};

		\foreach \a in {0,4,8}{
			\draw[edge] ({\a+2},2.9) to[bend left] ({-0.8},1.6);
			\draw[edge] (\a,2.9) to[bend left] (-0.8,1.6);
			\draw[edge][PastelOrange] ({\a+0.1},3) to ({\a+1.9},3);
			\draw[edge][PastelOrange] ({\a+2.1},3) to ({\a+3.9},3);
		
		}
		\foreach \a in {1,2,3,4,5,6}{
			
			\draw ({\a*2-2},3.5) node {$v_\a$};
			\draw[LavenderMagenta] ({\a*2-2},3) node {\huge.};
		}
		
		\draw[edge][white] ({10.1},3) to ({11.9},3);
		\draw[edge][PastelOrange, path fading=east] ({10.1},3) to ({11.9},3);
		
		\draw[edge, path fading=east] (12,2.9) to[bend left] (-0.8,1.6);
		
		\draw[PastelOrange] (12.5,3) node {$R$};
	\end{tikzpicture}
			\caption{A subgraph of type~\labelcref{itm:centre_3}.}
	\end{subfigure}\hfill

	\begin{subfigure}[b]{0.8\textwidth}
	\begin{tikzpicture}
		
		\foreach \a in {1,5, 9}{
			\draw[edge][CornflowerBlue] ({\a+0.1},0) to ({\a+1.9},0);
			\draw[edge][CornflowerBlue] ({\a+2},0.1) to ({\a+2},1.9);
			\draw[edge][PastelOrange, dashed] ({\a+2},0.1) to ({\a+2},1.9);
			\draw[edge][CornflowerBlue] ({\a+1.9},2) to ({\a+0.1},2);	
			\draw[edge][CornflowerBlue] (\a,1.9) to ({\a},0.1);
			\draw[edge][PastelOrange, dashed] (\a,1.9) to ({\a},0.1);		
			
			\draw[edge][PastelOrange] ({\a+3.9},0) to ({\a+2.1},0);
			\draw[edge][PastelOrange] ({\a+2.1},2) to ({\a+3.9},2);
			
			\draw[] (\a,0) node {\huge.};
			\draw[] (\a,2) node {\huge.};
			\draw[] (\a+2,0) node {\huge.};
			\draw[] (\a+2,2) node {\huge.};
			\draw[] (\a+4,0) node {\huge.};
			\draw[] (\a+4,2) node {\huge.};
		}
		\draw[edge][CornflowerBlue] (13,1.9) to ({13},0.1);
		\draw[edge][PastelOrange, dashed] (13,1.9) to ({13},0.1);
		
		\draw[edge][CornflowerBlue] (1,1.9) to (1,0.1);
		
		\draw[] ({1},2) node {\huge.};
		
		\foreach \a in {1,3,5}{
			\draw[CornflowerBlue] ({2*\a},1) node {$C_\a$};
		}
		
		\foreach \a in {2,4,6}{
			\draw[PastelOrange] ({2*\a},1) node {$C_\a$};
		}
		
		
		\draw[edge][CornflowerBlue, path fading=east] (15,2) to (13.1,2);
		\draw[edge][CornflowerBlue] (13.2,2) to (13.1,2);
		\draw[edge][CornflowerBlue, path fading=east] (13.1,0) to (15,0);

	\end{tikzpicture}
			\caption{A subgraph of type~\labelcref{itm:centre_4}.}
	\end{subfigure}
	\caption{Different types of subgraphs in~\cref{lemgroundwork2}.}
	\label{fig:groundwork}
\end{figure}

	\begin{proof}
		We assume without loss of generality that $U$ is countable.
		Let $T$ be an out-arborescence in $D$ that contains infinitely many elements of $U$ and whose leaves are contained in $U$, which exists by \cref{prop:spanning_arborescences}.
		
		We apply \cref{prop:star_comb_arborescences} to $T$ and $U$.
		If $T$ contains an out-star with leaves in $U$, then the root $r$ of this out-star is of type~\labelcref{itm:centre_1} and its branches form the desired $r$--$U$~paths that intersect only in $r$.
		Thus we can assume that $T$ contains an out-ray $R$ together with infinitely many disjoint (possibly trivial) directed $R$--$U$~paths $(P_i: i \in \NN)$.
		
		Since $D$ is strongly connected we can apply \cref{prop:in_tree_or_knit_ray} to $R$, which yields that there exists $r \in V(R)$ and either an in-arborescence $T'$ rooted at $r$ that intersects $V(R)$ exactly in $r$ and in infinitely many leaves of $T'$, or an in-ray $R'$ rooted at $r$ that is knit with $rR$.
		
		We begin with the former case and apply \cref{prop:star_comb_arborescences} to $T'$ and the set of leaves of $T'$ in $V(rR)$.
		If $T'$ contains a subdivided in-star $S$ with all leaves in $V(R) \cap V(T')$, let $c$ be the root of $S$.
		Then the concatenation of the directed $c$--$r$~path in $T'$ and the out-ray $rR$ forms together with the subdivided in-star $S$ a subgraph of type~\labelcref{itm:centre_3}.
		Otherwise, the in-arborescence $T'$ contains an in-ray $Q$ rooted at $r$ with infinitely many disjoint directed $rR$--$Q$~paths of $T'$.
		Note that $rR$ and $Q$ intersect only in $r$ since $T'$ intersects $rR$ only in leaves and in $r$.
		By omitting some of the directed $rR$--$Q$~paths, we ensure that the startvertices of the remaining directed $rR$--$Q$~paths appear in $\leq_R$-increasing order.
		Thus we obtained a subgraph of type~\labelcref{itm:centre_2}.
		
		We turn our attention to the latter case, i.e.\ $rR \cup R' = \bigcup_{i \in \NN} C_i$ for a sequence $(C_i)_{i \in \NN}$ of directed cycles with the property that for every $i \neq j \in \NN$ $C_i \cap C_j$ is a directed path if $|i-j|=1$ and empty otherwise.
		Then $rR \cup R'$ is a subgraph of type~\labelcref{itm:centre_4}.
		
		In both cases, let $A$ be the subgraph that we chose.
		It remains to define infinitely many disjoint directed $A$--$U$~paths.
		Since almost all directed paths $(P_i: i \in \NN)$ start in $V(rR)$, almost all directed paths $(P_i: i \in \NN)$ have a terminal segment that forms a directed $A$--$U$~path.
		Clearly, these terminal segments are disjoint and thus serve as the desired disjoint directed $A$--$U$~paths, which completes the proof.
	\end{proof}

\section{Forcing butterfly minors shaped by a star} \label{sec:forcing_stars}
In this section we investigate the structure of directed paths from $U$ to $A$ in the setting of~\cref{lemgroundwork2}, i.e.\ given a subgraph $A$ and infinitely many directed $A$--$U$~paths $(P_i: i \in \NN)$ that are disjoint, if $A$ is not a single vertex, or otherwise non-trivial and disjoint outside $A$.
We aim for~\cref{lem:star_or_back_path}, which shows that there exists either a butterfly minor shaped by a star with all teeth in $U$ or a family $(Q_i : i \in I)$ of directed paths from $U$ to $A$ for some infinite subset $I \subseteq \NN$ such that $Q_i$ starts in the endvertex of $P_i$ and $(P_i \cup Q_i) \cap (P_j \cup Q_j) = \emptyset$ for every $i \neq j \in I$.

The proof of~\cref{lem:star_or_back_path} can be outlined as follows:
Let $(Q_i: i \in \NN)$ be an infinite family of directed paths such that $Q_i$ starts in the endvertex of $P_i$ and ends in $A$.
If there is an infinite subset $J \subseteq \NN$ such that the directed paths in $(Q_j: j \in J)$ are disjoint, we construct the desired set $I \subseteq J$.
Otherwise, we will construct a butterfly minor shaped by a star, which is the main task in this section.

We begin by proving a sufficient condition for the existence of a butterfly minor shaped by a star.
\begin{lem}\label{lemhelp}
	Let $T$ be an out-arborescence and let $R$ be an in-ray with common root $r$ such that $R$ and $T$ intersect only in $r$ and in infinitely many leaves of $T$.
	Further, let $U$ be a set of vertices and let $(W_i: i \in \NN)$ be a family of disjoint directed $R$-$R$-walks that intersect $T$ only in $V(R)$ and each $W_i$ contains an element of $U$.
	Then there exists a butterfly-minor shaped by a star with all teeth in~$U$.
\end{lem}

\begin{figure}[ht]
	\begin{tikzpicture}

		\draw[edge][PastelOrange] (-0.05,2.95) to (-1,2.1);  
		
		\draw[LavenderMagenta] (-1,2) node {\huge.};
		\draw (-1,1.5) node {$r$};    
		
		\foreach \a in {0,4,8}{
			\draw[edge] ({\a+2},2.1) to ({\a+2},2.9);
			\draw[edge][PastelOrange] ({\a+1.9},3) to ({\a+0.1},3);
			\draw[edge][PastelOrange] ({\a+3.9},3) to ({\a+2.1},3);
		}
		
		\foreach \a in {0,4}{
			\draw[edge] ({-0.9},2) to[bend right] ({\a+2},1.9);
		}
		\draw[edge] (6.1,2) to (9.9,2);
		\draw[edge, path fading=east] (10.1,2) to (11.9,2);
		
		\foreach \a in {1,3,5}{
			\draw[] ({\a*2-2},3) node {\huge.};
			\draw[] ({\a*2},3) node {\huge.};
			\draw[] ({\a*2},2) node {\huge.};
		}
		
		\draw[edge] (2.1,2) to (4,2.9);
		
		\draw[edge][white] ({11.9},3) to ({10.1},3);
		\draw[edge][PastelOrange, path fading=east] ({11.9},3) to ({10.1},3);
		\draw[edge][PastelOrange] ({10.5},3) to ({10.1},3);
		
		\draw[edge][CornflowerBlue, bend left] (0,3.1) to (0.9,4.5);
		\draw[edge][CornflowerBlue, bend left] (1.1,4.5) to (2,3.1);
		\draw[CornflowerBlue] (1,4.5) node {\huge.};
		\draw[CornflowerBlue] (1,4.7) node {$u_1$};
		\draw[CornflowerBlue] (1,4) node {$W_1$};
		
		\draw[edge][CornflowerBlue, bend left] (5.1,4.5) to (6,4.1);
		\draw[edge][CornflowerBlue, bend left] (4,4.1) to (4.9,4.5);
		\draw[CornflowerBlue] (5,4.5) node {\huge.};
		\draw[CornflowerBlue] (5,4.7) node {$u_2$};
		\draw[CornflowerBlue] (5,3.5) node {$W_2$};
		\draw[CornflowerBlue] (4,4) node {\huge.};
		\draw[CornflowerBlue] (6,4) node {\huge.};
		\draw[edge][CornflowerBlue] (4,3.9) to (4,3.1);
		\draw[edge][CornflowerBlue] (6,3.1) to (6,3.9);
		\draw[edge][CornflowerBlue] (5.9,4) to (4.1,4);
		
		\draw[edge][CornflowerBlue] (8,3.1) to (8,4.4);
		\draw[edge][CornflowerBlue] (8,4.4) to (8,3.1);
		\draw[CornflowerBlue] (8,4.5) node {\huge.};
		\draw[CornflowerBlue] (8,4.7) node {$u_3$};
		\draw[CornflowerBlue] (8.5,4) node {$W_3$};
		
		\draw[PastelOrange] (11,2.5) node {$R$};
		\draw[] (4,2) node {$T$};
		
		\draw[] (-0.3,3.3) node {$a_1$};
		\draw[] (2.3,3.3) node {$b_1$};
		\draw[] (3.7,3.3) node {$b_2$};
		\draw[] (6.3,3.3) node {$a_2$};
		\draw[] (3.7,4.1) node {$a_2'$};
		\draw[] (6.3,4.1) node {$b_2'$};
		\draw[] (8.8,3.3) node {$a_3=b_3$};
		
	\end{tikzpicture}
	\caption{An in-ray $R$, an out-arborescence $T$ and directed walks $W_1, W_2, W_3$ as in \cref{lemhelp}. }
	\label{fig:lemhelp}
\end{figure}

\begin{proof}
	By \cref{prop:laced_paths} we can assume that $W_i$ is the union of a directed $R$--$U$~path $P_i$ and a directed path $Q_i$ starting in the endvertex of $P_i$ and intersecting $R$ only in its endvertex such that $P_i, Q_i$ are laced for every $i \in \NN$.
	Let $a_i$ be startvertex of $W_i$ and let $b_i$ be the endvertex of $W_i$.
	See~\cref{fig:lemhelp}.
	
	First, we pick an infinite subset $I \subseteq \NN$ such that for every $i < j \in I$ there exists a leave $t(i,j)$ of $T$ in $V(R)$ such that $a_i, b_i <_R t(i,j) <_R a_j, b_j$, which can be constructed recursively.
	For simplicity, we assume without loss of generality that $I = \NN$ and set $t_n := t(n, n+1)$.
	Further, let $T'$ be the out-arborescence in $T$ rooted at $r$ that has leaves exactly in $\{t_n: n \in \NN\}$.
	
	Second, we take a butterfly minor $D'$ of $D:= T' \cup R \cup \bigcup_{i \in \NN} W_i$ by deleting all edges and all internal vertices of $b_n R a_n$ for every $n \in \NN$ with $b_n <_R a_n$.
	Note that there exists a unique directed $a_n$--$b_n$~path $O_n$ in $P_n \cup Q_n$ for every $n \in \NN$ with $b_n <_R a_n$ since $P_n, Q_n$ are laced.
	Let $R'$ be the in-ray in $D'$ obtained from $R$ by replacing $b_n R a_n$ by $b_n O_n a_n$ for every $n \in \NN$ with $b_n <_R a_n$.
	Further, let $P_n'$ be the terminal segment of $P_n$ that intersects $R'$ only in its startvertex $a_n'$ and let $Q_n'$ be the terminal segment of $Q_n$ that intersects $R'$ only in its endvertex $b_n'$.
	Note that $a_n' \leq_{R'} b_n'$ for every $n \in \NN$ by construction of $R'$ and since $P_n$ and $Q_n$ are laced.
	
	Third, we take a butterfly minor $D''$ of $D'$ by butterfly contracting $a_n' R' t_n$ to a single vertex for every $n \in \NN$.
	Note that this is possible since all vertices of $\{v \in V(R'): a_n' <_{R'} v \leq_{R'} t_n \}$ have out-degree $1$ in $D'$.
	
	The directed graph $D''$ consists of an in-ray $R'$ and an out-arborescence $T'$ with common root $r$ that intersect only in $r$ and in infinitely many leaves of $T'$ together with directed $R'$--$U$~paths $(P_n')_{n \in \NN}$ and directed paths $(Q_n')_{n \in \NN}$ intersecting $V(R')$ only in its endvertex such that
	\begin{itemize}
		\item the startvertex of $P_n'$ coincides with the endvertex of $Q_n'$ and a leave of $T'$,
		\item the endvertex of $P_n'$ coincides with the startvertex of $Q_n'$ and is an element of $U$,
		\item $P_n'$ and $Q_n'$ intersect $T'$ only in $V(R')$, and
		\item $P_n'$ and $Q_n'$ are laced.
	\end{itemize}
	
	Finally, we butterfly contract $T'$ to an out-star, suppress all subdivision vertices of $R'$ and butterfly contract $P_i' \cup Q_i'$ to a double path for every $i \in \NN$, which is possible by \cref{obs:double_path}.
	This forms a butterfly minor shaped by a star with all teeth in $U$ whose centre is a dominated directed ray.
\end{proof}

\begin{lem}\label{lem:star_or_back_path}
	Let $D$ be a strongly connected directed graph and let $A$ be a strongly connected subgraph of $D$.
	Further, let $U \subseteq V(D)$ and let $(P_i: i \in \NN)$ be an infinite family of directed $A$--$U$~paths that are disjoint, if $A$ is not a single vertex, and otherwise are non-trivial and disjoint outside of $A$.
	Then there exists 
	either a butterfly minor of $D$ that is shaped by a star with all teeth in $U$
	or
	an infinite subset $I \subseteq \NN$ and a family $(Q_i: i \in I)$ of directed paths such that
	\begin{itemize}
	\item $Q_i$ starts in the endvertex of $P_i$ for every $i \in I$,
	\item  $Q_i$ has precisely its endvertex in $V(A)$ for every $i \in I$,
	\item $P_i$ and $Q_i$ are laced for every $i \in I$, and
	\item $(P_i \cup Q_i) \cap (P_j \cup Q_j) = \emptyset$ for every $i \neq j \in I$.
\end{itemize}
\end{lem}
\begin{proof}
    We can assume without loss of generality that $U$ is countable.
	By~\cref{prop:spanning_arborescences}, there exists an in-arborescence  $T$ in $D$ rooted in some vertex of $A$ containing $U$.
	Let $(Q_i: i \in \NN)$ be a family of directed paths in $T$ such that $Q_i$ starts in the endvertex of $P_i$ and intersects $A$ precisely in its endvertex.
	Note that every weak component of $\bigcup_{i \in \NN} Q_i$ is an in-arborescence, by construction.
	
	If $\bigcup_{i \in \NN} Q_i$ has infinitely many weak components, $A$ is not a single vertex and thus the directed paths $(P_i: i \in \NN)$ are disjoint.
	Furthermore, there exists an infinite subset $J \subseteq I$ such that the directed paths $(Q_i: i \in J)$ are disjoint.
	We construct the desired infinite subset $I \subseteq J$ recursively in the following way:
	Suppose that we have already constructed a finite set $I' \subseteq J$ such that $(P_i \cup Q_i) \cap (P_j \cup Q_j) = \emptyset$ for every $i \neq j \in I'$.
	Then there exists $j \in J \setminus I'$ such that $P_j$ and $Q_j$ avoid the finite subgraph $\bigcup_{i \in I'} P_i \cup Q_i$.
	We continue with $I' \cup \{j\}$.
	After countably many steps we obtain an infinite set $I \subseteq \NN$ such that $(P_i \cup Q_i) \cap (P_j \cup Q_j) = \emptyset$ for every $i \neq j \in I$.
	Applying \cref{prop:laced_paths} to $Q_i$ and $P_i$ gives a directed path $\Tilde{Q}_i$ with the same start- and endvertex as $Q_i$ which is laced with $P_i$ and $P_i \cup \Tilde{Q}_i \subseteq P_i \cup Q_i$.
	Thus $I$ and $(\Tilde{Q}_i: i \in I)$ are as desired.
	
	Thus we can assume that $\bigcup_{i \in \NN} Q_i$ has finitely many weak components.
	For simplicity, we assume without loss of generality that $\bigcup_{i \in \NN} Q_i$ is a single weak component.
	We show that there exists a butterfly minor of $D$ that is shaped by a star.
	
	Let $r$ be the root of the in-arborescence $\bigcup_{i \in \NN} Q_i$.
	Note that if $A$ is a single vertex, then $A=\{r\}$ and thus $r$ is the common startvertex of the directed paths $(P_i: i \in \NN)$.
	We apply \cref{prop:star_comb_arborescences} to $\bigcup_{i \in \NN} Q_i$ and the set of all startvertices of $(Q_i: i \in \NN)$, to obtain either a subdivided in-star $S$ whose leaves are in $U$ or an in-ray $R$ rooted in $r$ together with infinitely many disjoint directed $U$--$R$~paths~$(O_n: n \in \NN)$.
	Next, we show that in both cases there is either the desired butterfly minor shaped by a star with all teeth in $U$ or an infinite subset $J \subseteq \NN$ such that $(P_j \cap Q_k) \subseteq \{r\} $ for every $j \neq k \in J$.

	\begin{description}
		\item[If there exists a subdivided in-star $S$] See~\cref{fig:Tisstar}. Let $c$ be the root of the subdivided in-star~$S$ and let $P$ be the unique directed $c$--$r$~path in $\bigcup_{i \in \NN} Q_i$.
		For simplicity, we assume without loss of generality that all directed paths $(Q_i: i \in \NN)$ are contained in $S \cup P$.
		Since $P$ is finite, we can assume without loss of generality for simplicity that all directed paths $(P_i: i \in \NN)$ avoid $P -r$.
		We construct a strictly increasing sequence $(J(n))_{n \in \NN}$ of finite subsets of $\NN$ such that 
		$(P_j \cap Q_k) \subseteq \{r\} $ for every $n \in \NN$ and every $j \neq k \in J(n)$.
		
		Set $J(0):= \emptyset$ and let $n \in \NN$ be such that $J(n - 1)$ has been constructed.
		Since $\bigcup_{j \in J(n - 1)} P_j$ avoids $P - r$, there are infinitely directed paths in $(Q_i: i \in \NN)$ for which $Q_i -r$ avoids $\bigcup_{j \in J(n - 1)} P_j$.
		Since $\bigcup_{j \in J(n - 1)} Q_j -r \subseteq V(D) \setminus V(A)$ is finite, almost all directed paths in $(P_i: i \in \NN)$ avoid $\bigcup_{j \in J(n - 1)} Q_j -r $.
		Thus there exists $k \in \NN$ such that $P_k$ avoids all $(Q_j - r: j \in J(n - 1))$ and $Q_k -r $ avoids all $(P_j: j \in J(n - 1))$.
		We set $J(n):= J(n-1) \cup \{k\}$.
		After countably many steps we obtain an infinite set $J := \bigcup_{n \in \NN} J(n)$ with the property $(P_j \cap Q_k) \subseteq \{r\} $ for every $j \neq k \in J$.
		
			\begin{figure}[ht]
	\begin{tikzpicture}

		\node[draw, circle, minimum size=2cm, fill=lightgray] at (-1,3) {$A$};
		
		\foreach \a in {0}{
			\draw[edge][CornflowerBlue] ({\a+1.9},3) to ({\a+0.1},3);
			\draw[edge][CornflowerBlue] ({\a+3.9},3) to ({\a+2.1},3);
		}

		\foreach \a in {2}{
			\draw[] ({\a*2-2},3) node {\huge.};
			\draw[] ({\a*2},3) node {\huge.};
		}
		
		\foreach \a in {1,2,3,4}{
			\draw[edge, CornflowerBlue] ({4+1.9*sin(30*\a-30)},{3+1.9*cos(30*\a-30)}) to ({4+0.2*sin(30*\a-30)},{3+0.2*cos(30*\a-30)});
			\draw[] ({4+2*sin(30*\a-30)},{3+2*cos(30*\a-30)}) node {\huge.};
			\draw[CornflowerBlue] ({4+2.5*sin(30*\a-30)},{3+2.5*cos(30*\a-30)}) node {$Q_\a$};
		}
		
		\draw[edge, CornflowerBlue, opacity=0.5] ({4+1.9*sin(30*5-30)},{3+1.9*cos(30*5-30)}) to ({4+0.2*sin(30*5-30)},{3+0.2*cos(30*5-30)});
		\draw[opacity=0.5] ({4+2*sin(30*5-30)},{3+2*cos(30*5-30)}) node {\huge.};
		\draw[CornflowerBlue, opacity=0.5] ({4+2.5*sin(30*5-30)},{3+2.5*cos(30*5-30)}) node {$Q_5$};
		
		\draw[edge, CornflowerBlue, opacity=0.3] ({4+1.9*sin(30*6-30)},{3+1.9*cos(30*6-30)}) to ({4+0.2*sin(30*6-30)},{3+0.2*cos(30*6-30)});
		\draw[opacity=0.3] ({4+2*sin(30*6-30)},{3+2*cos(30*6-30)}) node {\huge.};
		\draw[CornflowerBlue, opacity=0.3] ({4+2.5*sin(30*6-30)},{3+2.5*cos(30*6-30)}) node {$Q_6$};

		\draw[CornflowerBlue] (3.7,2.7) node {$c$};

		\draw[] (0,3) node {\huge.};
		\draw[] (0.1,2.7) node {$r$};
		
	\end{tikzpicture}
	\caption{The structure of $A \cup S \cup P$ if there exists a subdivided in-star $S$ in the proof of~\cref{lem:star_or_back_path}.}
	\label{fig:Tisstar}
	
\end{figure}
		
		\item[If there exists an in-ray $R$ with directed $U$--$R$~paths $(O_n: n \in \NN)$]
		See~\cref{fig:Tiscomb}.
		We assume without loss of generality for simplicity that $Q_n \subseteq O_n \cup R$ for every $n \in \NN$, i.e.\ $O_n$ is an initial segment of $Q_n$.
		We construct a strictly increasing sequence $(J(n))_{n \in \NN}$ of finite subsets of $\NN$ such that $(P_j \cap Q_k) \subseteq R$ for every $n \in \NN$ and every $j \neq k \in J(n)$.
		
		Set $J(0):= \emptyset$ and let $n \in \NN$ be such that $J(n - 1)$ has been constructed.
		Since $\bigcup_{j \in J(n-1)} P_j$ is finite almost all elements of $(O_n: n \in \NN)$ are disjoint to $\bigcup_{j \in J(n-1)} P_j$.
		Furthermore, almost all directed paths $(P_n: n \in \NN)$ are disjoint to $\bigcup_{j \in J(n-1)} Q_j -r $ since $\bigcup_{j \in J(n-1)} Q_j - r \subseteq V(D) \setminus V(A)$ is finite.
		Thus there exists $k \in \NN$ such that $Q_k$ intersects $(P_j: j \in J(n - 1))$ only in $R$, and $P_k$ intersects $\bigcup_{j \in J(n - 1)} Q_j$ at most in $r$.
		We set $J(n):= J(n-1) \cup \{k\}$.
		After countably many steps we obtain an infinite set $J := \bigcup_{n \in \NN} J(n)$ with the property that
		$(P_j \cap Q_k) \subseteq R $ for every $j \neq k \in J$.
		
		By Ramsey's theorem there exists an infinite subset $J' \subseteq J$ such that either $(P_j \cap Q_k) \subseteq \{r\}$ for every $j \neq k \in J'$ or $(P_j \cap Q_k) \not\subseteq \{r\}$ for every $j \neq k \in J'$.
		In the former case, $J'$ is as desired.
		
		In the latter case, i.e.\ each directed path $(P_j: j \in J')$ intersects $R -  r$, we will apply~\cref{lemhelp}.
		We partition $J'$ into two infinite set $K$ and $L$.
		Let $\hat{T}$ be an out-arborescence in $A \cup \bigcup_{k \in K} {P}_k$ rooted at $r$ with infinitely many leaves in $R$.
		Furthermore, let $\Tilde{P}_\ell$ be the unique terminal segment of $P_\ell$ that starts in $R$ and is otherwise disjoint to $R$ for every $\ell \in L$.
		Note that the concatenation $ \Tilde{P}_\ell \circ O_\ell$ forms a directed $R$--$R$~walk that contains some element of $U$.
		Moreover, $ \Tilde{P}_\ell \circ O_\ell$ intersects $\hat{T}$ only in $V(R)$ for every $\ell \in L$ since $\Tilde{P}_\ell$ is disjoint to $\hat{T}$ by definition of $\hat{T}$ and as $O_\ell \cap \bigcup_{k \in K} P_k \subseteq Q_\ell \cap \bigcup_{k \in K} P_k \subseteq R$ by choice of $J$.
		
		Since the directed paths $(\Tilde{P}_\ell: \ell \in L)$ are disjoint and the directed paths $(O_\ell: \ell \in L)$ are disjoint, each vertex of $D$ can be contained in at most two elements of $(\Tilde{P}_\ell \circ O_\ell: \ell \in L)$.
		Thus we can define recursively an infinite set $L' \subseteq L$ such that $(\Tilde{P}_\ell \circ O_\ell: \ell \in L')$ is a family of disjoint directed walks.
		Thus $\hat{T}, R, r, U$ and $(\Tilde{P}_\ell \circ O_\ell: \ell \in L')$ satisfy the conditions of~\cref{lemhelp}, which shows that there exists a butterfly minor shaped by a star with all teeth in $U$.

			\begin{figure}[ht]
	\begin{tikzpicture}

		\node[draw, circle, minimum size=2cm, fill=lightgray] at (-1,3) {$A$};
		
		\foreach \a in {0,4}{
			\draw[edge][CornflowerBlue] ({\a+1.9},3) to ({\a+0.1},3);
			\draw[edge][CornflowerBlue] ({\a+3.9},3) to ({\a+2.1},3);
		}
		
		\foreach \a in {2,4}{
			\draw[] ({\a*2-2},3) node {\huge.};
			\draw[] ({\a*2},3) node {\huge.};
		}
		
		\draw[edge][CornflowerBlue, path fading=east] ({9.9},3) to ({8.1},3);
		\draw[edge][CornflowerBlue] ({8.4},3) to ({8.1},3);
		
		\foreach \a in {1,2,3,4}{
			\draw[edge, CornflowerBlue] ({2*\a},4.9) to ({2*\a},3.1);
			\draw[] ({2*\a},5) node {\huge.};
			\draw[CornflowerBlue] ({2*\a-0.5},4) node {$O_\a$};
		}
		
		\draw[CornflowerBlue] (9,2.5) node {$R$};

		\draw[] (0,3) node {\huge.};
		\draw[] (0.1,2.7) node {$r$};
		
	\end{tikzpicture}
	\caption{The structure of $A \cup R \cup \bigcup_{n \in \NN} O_n$ if there exists an in-ray $R$ with directed $U$--$R$~paths $(O_n: n \in \NN)$ in the proof of~\cref{lem:star_or_back_path}.}
	\label{fig:Tiscomb}
\end{figure}
		
	\end{description}
	
	Finally, we can assume that there exists an infinite subset $J \subseteq \NN$ such that $(P_j \cap Q_k) \subseteq \{r\} $ for every $j \neq k \in J$.
	It remains to show that there exists a butterfly minor shaped by a star with all teeth in $U$.
	
	We apply \cref{prop:laced_paths} to $Q_j$ and $P_j$ to obtain a directed path $\Tilde{Q}_j$ that is laced with $P_j$, has the same startvertex and endvertex as $Q_j$ and $\Tilde{Q}_j \subseteq Q_j \cup P_j$ for every $j \in J$.
	See~\cref{fig:tilde}.
	We argue that $\Tilde{Q}_j$ intersects $A$ only in its last vertex $r$ and that $\Tilde{Q}_j \cap P_k \subseteq \{r\}$ for every $j \neq k \in J$.
	
	If $A$ is a single vertex, then $A=\{r\}$, which implies that $\Tilde{Q}_j$ intersects $A$ only in its last vertex.
	Furthermore, $r$ is the common startvertex of all directed paths $(P_k: k \in J)$, which implies $\Tilde{Q}_j \cap P_k \subseteq (Q_j \cup P_j) \cap P_k \subseteq \{r\}$ for every $j \neq k \in J$.
	
	Otherwise, the directed paths $(P_k : k \in J)$ are disjoint and we can assume without loss of generality that no directed path in $(P_k : k \in J)$ contains $r$.
	Then $Q_j \cap P_k = \emptyset$ for every $j \neq k \in J$, which implies $\Tilde{Q}_j \cap P_k \subseteq (Q_j \cup P_j) \cap P_k = \emptyset$ for every $j \neq k \in J$.
	Since $Q_k$ intersects $A$ only in $r \notin P_k$, the startvertex of $P_k$, which is an element of $A$, is not contained in $Q_k$ and thus has in-degree $0$ in $P_k \cup Q_k$. Therefore, $\Tilde{Q}_k$ does not contain the startvertex of $P_k$, which implies that $\Tilde{Q}_k$ intersects $A$ only in $r$.
	
	We show that $\bigcup_{j \in J} \Tilde{Q}_j$ is an in-arborescence:
	Let $v \in V(\bigcup_{j \in J} \Tilde{Q}_j) \setminus \{r\} \subseteq V(D) \setminus V(A)$ be some vertex that is contained in $\Tilde{Q}_j$ and $\Tilde{Q}_k$ for some $j, k \in J$.
	Since $V(Q_k \cap P_j) \setminus \{r\} = \emptyset$, $V(Q_j \cap P_k) \setminus \{r\} = \emptyset$ and $V(P_k \cap P_j) \subseteq V(A)$, $v \in V(\Tilde{Q}_j \cap \Tilde{Q}_k) \setminus \{r\} \subseteq V(Q_j \setminus P_j) \cap V(Q_k \setminus P_k)$.
	We can deduce that every edge of $\Tilde{Q}_j$ with tail in $v$ is contained in $Q_j$ and every edge of $\Tilde{Q}_k$ with tail in $v$ is contained in $Q_k$.
	Since $\bigcup_{j \in J} Q_j$ is an in-arborescence, exactly one edge of $Q_j \cup Q_k$ has tail in $v$.
	Thus exactly one edge of $\Tilde{Q}_j \cup \Tilde{Q}_k$ has tail $v$, which implies that $\bigcup_{j \in J} \Tilde{Q}_j$ is an in-arborescence.
	
	Let $T$ be some (possibly trivial) out-arborescence in $A$ rooted at $r$ that contains all startvertices of directed paths $(P_j: j \in J)$.
	Further, let $\hat{Q}$ be the subgraph of $\bigcup_{j \in J} \Tilde{Q}_j$ induced by all vertices that are contained in at least two elements of $(\Tilde{Q}_j: j \in J)$.
	Note that $\hat{Q}$ is an in-arborescence rooted at $r$ that intersects $\bigcup_{j \in J} P_j$ at most in $r$.
	We consider $T \cup  \bigcup_{j \in J} P_j  \cup \bigcup_{j \in J}  \Tilde{Q}_j$ and butterfly contract $T \cup \hat{Q}$ to a single vertex $r$.
	Further, we butterfly contract $P_j$ and $\Tilde{Q}_j \setminus \hat{Q}$ to a double path for every $j \in J$, which is possible by~\cref{obs:double_path}.
	The obtained directed graph is shaped by a star and has all teeth in~$U$.
\end{proof}

\begin{figure}
	\centering
	\begin{subfigure}{\textwidth}
		\centering
	\begin{tikzpicture}
		
		\draw[black] (2,1) node {\huge.};
		
		\foreach \a in {4,...,8}{
			\draw[edge][PastelOrange] ({\a+0.1},1) to ({\a+0.9},1);
			\draw[black] (\a,1) node {\huge .};
			\draw[black] ({\a+1},1) node {\huge .};
		}
		
		\draw[edge][CornflowerBlue] (4,1.1) to[bend right] (2.1,2);
		\draw[edge][CornflowerBlue] (8,1.1) to[bend right] (5,1.1);
		\draw[edge][CornflowerBlue] (6,0.9) to[bend left] (4,0.9);
		\draw[edge][CornflowerBlue] (9,0.9) to[bend left] (7,0.9);
		\draw[edge][CornflowerBlue] (0,1.1) to[bend left] (1.9,2);
		\draw[edge][CornflowerBlue] (2,1.9) to (2,1.1);

		\draw (2.1,0.6) node {$A = \{r\}$};

		\draw[black] (0,1) node {\huge.};
		\draw[black] (-1,1) node {\huge.};
		\draw[black] (2,2) node {\huge.};
		
		\draw[edge][PastelOrange] (2.1,1) to (3.9,1);
		\draw[edge][PastelOrange] (1.9,1) to (0.1,1);
		\draw[edge][PastelOrange, bend left] (-0.1,1) to (-0.9,1);
		\draw[edge, CornflowerBlue, bend left] (-0.9,1) to (-0.1,1);
		
		\draw[edge, CornflowerBlue, dashed] (5.2,1) to (5.9,1);
		\draw[edge, CornflowerBlue, dashed] (7.2,1) to (7.9,1);
		
		\draw[CornflowerBlue] (8,2) node {$\Tilde{Q}_k$};
		\draw[CornflowerBlue] (-0.3,2) node {$\Tilde{Q}_j$};
		\draw[PastelOrange] (7,0.5) node {$P_k$};
		\draw[PastelOrange] (-0.3,0.5) node {$P_j$};
	\end{tikzpicture}
		\caption{If the subgraph $A$ is trivial:
		The directed $A$--$U$~paths $P_j$ and $P_k$ intersect only in $r$. The directed paths $\Tilde{Q}_j$ and $\Tilde{Q}_k$ have a terminal segment in common and are laced with $P_j$ and $P_k$, respectively. Moreover, $P_j$ intersects $\Tilde{Q}_k$ only in $r$ and $P_k$ intersects $\Tilde{Q}_j$ only in $r$.
	}
	\end{subfigure}
	
	\begin{subfigure}{\textwidth}
		\centering
		\begin{tikzpicture}
			
			\node[draw, circle, minimum size=2cm, fill=lightgray] at (2,1) {$A$};
			
			\foreach \a in {3,...,8}{
				\draw[edge][PastelOrange] ({\a+0.1},1) to ({\a+0.9},1);
				\draw[black] (\a,1) node {\huge .};
				\draw[black] ({\a+1},1) node {\huge .};
			}
			
			\draw[edge][CornflowerBlue] (4,1.1) to[bend right] (2.1,3);
			\draw[edge][CornflowerBlue] (8,1.1) to[bend right] (5,1.1);
			\draw[edge][CornflowerBlue] (6,0.9) to[bend left] (4,0.9);
			\draw[edge][CornflowerBlue] (9,0.9) to[bend left] (7,0.9);
			\draw[edge][CornflowerBlue] (0,1.1) to[bend left] (1.9,3);
			\draw[edge][CornflowerBlue] (2,2.9) to (2,2);

			\draw (1.8,2.2) node {$r$};
			
			\draw[black] (2,2) node {\huge.};
			\draw[black] (0,1) node {\huge.};
			\draw[black] (-1,1) node {\huge.};
			\draw[black] (2,3) node {\huge.};
			\draw[black] (1,1) node {\huge.};
			
			\draw[edge][PastelOrange] (0.9,1) to (0.1,1);
			\draw[edge][PastelOrange, bend left] (-0.1,1) to (-0.9,1);
			\draw[edge, CornflowerBlue, bend left] (-0.9,1) to (-0.1,1);
			
			\draw[edge, CornflowerBlue, dashed] (5.2,1) to (5.9,1);
			\draw[edge, CornflowerBlue, dashed] (7.2,1) to (7.9,1);
			
			\draw[CornflowerBlue] (8,2) node {$\Tilde{Q}_k$};
			\draw[CornflowerBlue] (-0.3,2) node {$\Tilde{Q}_j$};
			\draw[PastelOrange] (7,0.5) node {$P_k$};
			\draw[PastelOrange] (-0.3,0.5) node {$P_j$};
		\end{tikzpicture}
				\caption{If the subgraph $A$ is non-trivial:
			The directed $A$--$U$~paths $P_j$ and $P_k$ are disjoint. The directed paths $\Tilde{Q}_j$ and $\Tilde{Q}_k$ have a terminal segment in common and are laced with $P_j$ and $P_k$, respectively. Moreover, $P_j$ avoids $\Tilde{Q}_k$ and $P_k$ avoids $\Tilde{Q}_j$.
		}
	\end{subfigure}
	
	\caption{The structure in the proof of~\cref{lem:star_or_back_path}.
	}
	\label{fig:tilde}
\end{figure}

\section{Proof of \cref{main_theorem}}\label{sec:proof}
We turn our attention to the proof of the main theorem:
\infinite*
	\begin{proof}
		By \cref{lemgroundwork2}, there exist a subgraph $A$ of type~\labelcref{itm:centre_1,itm:centre_2,itm:centre_3,itm:centre_4} and infinitely many directed $A$--$U$~paths that are disjoint, if $A$ is of type~\labelcref{itm:centre_2,itm:centre_3,itm:centre_4}, or non-trivial and disjoint outside $A$, if $A$ is of type~\labelcref{itm:centre_1}.
		By \cref{lem:star_or_back_path}, there exists either a butterfly minor shaped by a star with all leaves in~$U$ or a family $(P_i: i \in \NN)$ of directed $A$--$U$~paths and a family $(Q_i: i \in \NN)$ of directed paths such that
		\begin{itemize}
			\item $Q_i$ starts in the endvertex of $P_i$ for every $i \in \NN$,
			\item  $Q_i$ has precisely its endvertex in $V(A)$ for every $i \in \NN$,
			\item $P_i$ and $Q_i$ are laced for every $i \in \NN$, and
			\item $(P_i \cup Q_i) \cap (P_j \cup Q_j) = \emptyset$ for every $i \neq j \in \NN$.
		\end{itemize}
		Since we are done in the former case, we can assume the latter case.
		For each $i \in \NN$, let $a_i$ be the startvertex of $P_i$, let $u_i$ be the endvertex of $P_i$ and let $b_i$ be the endvertex of $Q_i$.
		We remark that the vertices of $(a_i: i \in \NN)$ are distinct and the vertices of $(b_i: i \in \NN)$ are distinct.
		Note that $A$ is not of type~\labelcref{itm:centre_1} witnessed by the infinitely many elements $(a_i: i \in \NN)$ in $A$.
		
		We consider the subgraph $D':= A \cup \bigcup_{i \in \NN} (P_i \cup Q_i)$.
		Let $D''$ be the butterfly minor of $D'$ obtained by butterfly contracting each $P_i \cup Q_i$ to a double path $F_i$ and edges $(a_i, x_i), (x_i, b_i)$, where $F_i$ has endpoints $x_i$ and $u_i$. 
		This is possible by~\cref{obs:double_path}.
		Now we construct the desired butterfly minor depending on the type of $A$ and the distribution of $(a_i: i \in \NN)$ and $(b_i: i \in \NN)$ in $A$.
		\begin{description}
			\item[If $A$ is of type~\labelcref{itm:centre_2}] Since the vertices of $(a_i: i \in \NN)$ are distinct and the vertices of $(b_i: i \in \NN)$ are distinct, there exists an infinite subset $I \subseteq \NN$ such that for every $i < j \in I$ there exists $f(i,j) \in \NN$ such that the vertices $a_i$ and $b_i$ are contained in the weak component of $A - P_{f(i,j)}$ containing $c$, and $a_j$ and $b_j$ are contained in the other weak component of $A - P_{f(i,j)}$.
			This implies that for every $i \neq j \in I$, $a_i$ and $b_j$ are not both contained in some directed path $P_n$.
			For simplicity, we assume without loss of generality that $I = \NN$.
			
			If there exists an infinite subset $J \subseteq \NN$ such that for each $j \in J$ there is some $n \in \NN$ such that $a_j$ and $b_j$ are internal vertices of $P_n$ with $b_j \leq_{P_n} a_j$, then we can butterfly contract $b_j P_n a_j$ so that $b_j = a_j$ for every $j \in J$.
			Then there exist an in-tree $T^{\mathrm{in}}$ and an out-tree $T^{\mathrm{out}}$ in $A$ that intersect only in their common root $u$ and in the vertices $a_j = b_j$ for every $j \in J$.
			See~\cref{fig:bleqa}.
			We consider $T^{\mathrm{in}} \cup T^{\mathrm{out}} \cup \bigcup_{j \in J} (F_j+(a_j, x_j) + (x_j, b_j))$ and butterfly contract $T^{\mathrm{in}}$ to an in-star and $T^{\mathrm{out}}$ to an out-star.
			This butterfly minor is shaped by a star with all teeth in $U$, as desired.
			
				\begin{figure}[ht]
				\begin{tikzpicture}
					
					\draw[edge][LavenderMagenta] (-0.95,2.5) to (-0.05,3.95);
					\draw[edge][CornflowerBlue] (0,1) to (-0.95,2.45);
					\draw[LavenderMagenta] (-1,2.5) node {\huge.}; 
					\draw (-1,2) node {$u$};
					
					\foreach \a in {0,4}{
						\draw[edge][LavenderMagenta] ({\a+2},3.95) to ({\a+2},3.05);
						\draw[edge][LavenderMagenta] (\a,3.95) to (\a,3.05); 
						\draw[edge][LavenderMagenta] ({\a+0.05},4) to ({\a+1.95},4);
						\draw[edge][LavenderMagenta] ({\a+2.05},4) to ({\a+3.95},4);
						\draw[edge][CornflowerBlue] ({\a+1.95},1) to ({\a+0.05},1);
						\draw[edge][CornflowerBlue] ({\a+3.95},1) to ({\a+2.05},1);
						\draw[edge][CornflowerBlue] ({\a},1.95) to ({\a},1.05);
						\draw[edge][CornflowerBlue] ({\a+2},1.95) to ({\a+2},1.05);
						\draw[edge][dashed] ({\a},2.95) to ({\a},2.05);
						\draw[edge][dashed] ({\a+2},2.95) to ({\a+2},2.05);
						\draw[edge][PastelOrange] ({\a+0.05},2.05) to ({\a+0.5},2.45);
						\draw[edge][PastelOrange] ({\a+0.5},2.55) to ({\a+0.05},2.95);
						\draw[PastelOrange] ({\a+0.5},2.5) node {\huge.};
						\draw[edge][PastelOrange] ({\a+2.05},2.05) to ({\a+2.5},2.45);
						\draw[edge][PastelOrange] ({\a+2.5},2.55) to ({\a+2.05},2.95);
						\draw[PastelOrange] ({\a+2.5},2.5) node {\huge.};
					}
					\foreach \a in {1,2,3,4}{
						\draw[LavenderMagenta] ({\a*2-2},3) node {\huge.};
						\draw[LavenderMagenta] ({\a*2-2},2) node {\huge.};
						\draw ({\a*2-1.5},3) node {$b_\a$};
						\draw ({\a*2-1.5},2) node {$a_\a$};
						\draw ({\a*2-1},2.5) node {$x_\a$}; 
						\draw[black] ({\a*2-2},1) node {\huge.};
					}
	
					\draw[edge][white] (6.05,4) to (7.95,4);				
					\draw[edge][LavenderMagenta, path fading=east] (6.05,4) to (8,4);

					\draw[][white] (8,1) to (6.15,1);				
					\draw[edge][CornflowerBlue, path fading=east] (8,1) to (6.05,1);

					\draw[CornflowerBlue] (8.5,1) node {$T^{\mathrm{in}}$};
					\draw[LavenderMagenta] (8.5,4) node {$T^{\mathrm{out}}$};
					
				\end{tikzpicture}
				\caption{The out-arborescence $T^{\mathrm{out}}$ and the in-arborescence $T^{\mathrm{in}}$ in $R_1 \cup R_2 \cup \bigcup_{n \in \NN} P_n \cup \bigcup_{i \in J} ((a_i, x_i) + (x_i, b_i))$ if $b_i \leq_{P_n} a_i$ for all $i \in J$ in the proof of~\cref{main_theorem}. The dashed edges will be butterfly contracted.				
				}
				\label{fig:bleqa}
			\end{figure}
			
			Thus we can assume that there exists an infinite subset $J \subseteq \NN$ such that $a_j$ and $b_j$ are internal vertices of the same path $P_n$ only if $a_j <_{P_n} b_j$.
			We butterfly contract $P_n a_j$ to its startvertex for every $a_j$ that is internal vertex of some $P_n$.
			Similarly, we butterfly $b_j P_n$ to its endvertex for every $b_j$ that is internal vertex of some $P_n$.
			Thus we can assume that $a_j$ and $b_j$ are contained in $V(R_1) \cup V(R_2)$.
			
			Then there exists an infinite subset $J' \subseteq J$ and $\epsilon, \delta \in \{1,2\}$ such that $a_j \in V(R_\epsilon)$ and $b_j \in V(R_\delta)$ for every $j \in J'$.
			If $\epsilon = \delta$, then we are done by \cref{lemhelp} and~\cref{obs:reversing}.
			See~\cref{fig:oneside}.
			
							\begin{figure}[ht]
\begin{tikzpicture}
	
	\draw[edge][LavenderMagenta] (-0.95,2.5) to (-0.05,3.45);
	\draw[edge][CornflowerBlue] (0,1.5) to (-0.95,2.45);
	\draw[LavenderMagenta] (-1,2.5) node {\huge.};
	\draw (-1,2) node {$u$};
	
	\foreach \a in {0,4}{
		\draw[edge][LavenderMagenta] ({\a+2},3.5) to ({\a+2},1.55);
		\draw[edge][LavenderMagenta] (\a,3.5) to (\a,1.55);
		\draw[edge][LavenderMagenta] ({\a+0.05},3.5) to ({\a+1.95},3.5); 
		\draw[edge][LavenderMagenta] ({\a+2.05},3.5) to ({\a+3.95},3.5); 
		\draw[edge][CornflowerBlue] ({\a+1.95},1.5) to ({\a+0.05},1.5); 
		\draw[edge][CornflowerBlue] ({\a+3.95},1.5) to ({\a+2.05},1.5);
	}
	\foreach \a in {1,2,3,4}{
		\draw[black] ({\a*2-2},3.5) node {\huge.};
		\draw[black] ({\a*2-2},1.5) node {\huge.};
	}
	
	\draw[edge][white] (6.05,3.5) to (7.95,3.5);				
	\draw[edge][LavenderMagenta, path fading=east] (6.05,3.5) to (8,3.5);
	
	\draw[][white] (8,1.5) to (6.15,1.5);				
	\draw[edge][CornflowerBlue, path fading=east] (8,1.5) to (6.05,1.5);
	
	\draw[CornflowerBlue] (8.5,1.5) node {$R_2$};
	\draw[LavenderMagenta] (8.5,3.5) node {$T^{\mathrm{out}}$};
	
\end{tikzpicture}
					\caption{The out-arborescence $T^{\mathrm{out}}$ in $R_1 \cup R_2 \cup \bigcup_{n \in \NN} P_n$ if $\delta = \epsilon = 2$ in the proof of~\cref{main_theorem}.}
				\label{fig:oneside}
			\end{figure}
			
			If $\epsilon = 1$, $\delta = 2$, $R_1$  is an out-arborescence that contains $(a_j: j \in J')$ and $R_2$ is an out-arborescence that contains $(b_j: j \in J')$.
			See~\cref{fig:twosides}.
			Note that $R_1$ and $R_2$ intersect only in their common root $u$.
			By restricting to $R_2 \cup R_1 \cup \bigcup_{j \in J'} (F_j+(a_j, x_j) + (x_j, b_j))$ and butterfly contracting $R_2 \cup R_1$ to a single vertex, we obtain a butterfly minor that is shaped by a star with all teeth in $U$, as desired.
			
										\begin{figure}[ht]
				\begin{tikzpicture}
					
					\draw[edge][LavenderMagenta] (-0.95,2.5) to (-0.05,3.45);
					\draw[edge][CornflowerBlue] (0,1.5) to (-0.95,2.45);
					\draw[LavenderMagenta] (-1,2.5) node {\huge.};
					\draw (-1,2) node {$u$};
					
					\foreach \a in {0,2,4}{
						\draw[edge][PastelOrange] ({\a},3.45) to ({\a+0.45},2.55); 
						\draw[edge][PastelOrange] ({\a+0.45},2.45) to ({\a},1.55); 
						\draw[edge][PastelOrange] ({\a+2},3.45) to ({\a+2.45},2.55); 
						\draw[edge][PastelOrange] ({\a+2.45},2.45) to ({\a+2},1.55); 
						\draw[edge][LavenderMagenta] ({\a+0.05},3.5) to ({\a+1.95},3.5); 
						\draw[edge][CornflowerBlue] ({\a+1.95},1.5) to ({\a+0.05},1.5);
					}
					\foreach \a in {1,2,3,4}{
						\draw[black] ({\a*2-2},3.5) node {\huge.}; 
						\draw[black] ({\a*2-2},1.5) node {\huge.}; 
						\draw[PastelOrange] ({\a*2-1.5},2.5) node {\huge.}; 
						\draw[black] ({\a*2-1},2.5) node {$x_\a$};
					}
					
					\draw[edge][LavenderMagenta, path fading=east] ({6.05},3.5) to ({7.95},3.5);
					\draw[edge][CornflowerBlue, path fading=east] ({7.95},1.5) to ({6.05},1.5); 
					\draw[edge][CornflowerBlue] ({6.45},1.5) to ({6.05},1.5); 
					
					\draw[CornflowerBlue] (8.5,1.5) node {$T^{\mathrm{in}}$}; 
					\draw[LavenderMagenta] (8.5,3.5) node {$T^{\mathrm{out}}$}; 
					
				\end{tikzpicture}
				\caption{The out-arborescence $T^{\mathrm{out}}$ and the in-arborescence $T^{\mathrm{in}}$  in $R_1 \cup R_2 \cup \bigcup_{i \in J'} ( (a_i, x_i) + (x_i, b_i))$ if $\epsilon =1, \delta = 2$ in the proof of~\cref{main_theorem}.}
				\label{fig:twosides}
			\end{figure}
			
			If $\epsilon = 2$ and $\delta =1$, we set $\beta(j):=f(j, j+1)$ and let $c_{\beta(j)}$ be the startvertex and $d_{\beta(j)}$ be the endvertex of $P_{\beta(j)}$ for every $j \in J'$.
			We consider $D''':=R_1 \cup R_2 \cup \bigcup_{j \in J'} P_{\beta(j)} \cup \bigcup_{j \in J'} (F_j+(a_j, x_j)+(x_j, b_j))$.
			See~\cref{fig:constructing_chain_of_triangles}.
			By butterfly contracting $b_j R_1 c_{\beta(j)}$ and $a_j R_2 d_{\beta(j)}$ to single vertices for every $j \in J'$, we obtain that $b_j = c_{\beta(j)}$ and $a_j = d_{\beta(j)}$ for every $j \in J'$.
			Furthermore, by suppressing all subdivision vertices of $D'''$ and deleting $u$, we obtain a chain of triangles with all teeth in $U$, as desired.
			
							\begin{figure}[ht]
				\begin{tikzpicture}
					
					\draw[edge][black] (-0.95,2.5) to (-0.05,3.45);
					\draw[edge][black] (0,1.5) to (-0.95,2.45);
					\draw[LavenderMagenta] (-1,2.5) node {\huge.};
					\draw (-1,2) node {$u$};
					
					\foreach \a in {0,4}{
						\draw[edge][PastelOrange] ({\a+0.45},2.55) to ({\a},3.45); 
						\draw[edge][PastelOrange] ({\a},1.55) to ({\a+0.45},2.45); 
				
						\draw[edge][black, dashed] ({\a+0.05},3.5) to ({\a+1.95},3.5); 
						\draw[edge][black, dashed] ({\a+1.95},1.5) to ({\a+0.05},1.5); 
					}
					\foreach \a in {1,2}{
						\draw[black] ({\a*4-4},3.5) node {\huge.}; 
						\draw[black] ({\a*4-4},1.5) node {\huge.};
						\draw[PastelOrange] ({\a*4-3.5},2.5) node {\huge.};
						\draw[black] ({\a*4-4},2.5) node {$x_\a$};
					}
					
					\draw[edge][black] (2.05, 3.5) to (3.95, 3.5);
					\draw[edge][black] (3.95, 1.5) to (2.05, 1.5);
					
					\draw[edge][black, path fading=east] ({6.05},3.5) to ({7.95},3.5);
					\draw[edge][black, path fading=east] ({7.95},1.5) to ({6.05},1.5);
					\draw[edge][black] ({6.45},1.5) to ({6.05},1.5);
					
					\draw[edge, CornflowerBlue] (2, 3.45) to (2,1.55);
					\draw[edge, CornflowerBlue] (6, 3.45) to (6,1.55);
					
					\draw[black] (8.5,1.5) node {$R_2$};
					\draw[black] (8.5,3.5) node {$R_1$};
					
					\draw[black] (2,1.5) node {\huge.};
					\draw[black] (2,3.5) node {\huge.};
					\draw[black] (6,1.5) node {\huge.};
					\draw[black] (6,3.5) node {\huge.};
					
					\draw[black] (0,1) node {$a_1$};
					\draw[black] (0,4) node {$b_1$};
					\draw[black] (2,1) node {$d_{\beta(1)}$};
					\draw[black] (2,4) node {$c_{\beta(1)}$};
					\draw[black] (4,1) node {$a_2$};
					\draw[black] (4,4) node {$b_2$};
					\draw[black] (6,1) node {$d_{\beta(2)}$};
					\draw[black] (6,4) node {$c_{\beta(2)}$};
					
				\end{tikzpicture}
					\caption{The structure of $R_1 \cup R_2 \cup \bigcup_{j \in J'} P_{\beta(j)} \cup \bigcup_{j \in J'} ((a_j, x_j) + (x_j, b_j))$ if $\epsilon=2, \delta=1$ in the proof of~\cref{main_theorem}. The dashed edges will be butterfly contracted.}
				\label{fig:constructing_chain_of_triangles}
			\end{figure}

			\item[If $A$ is of type~\labelcref{itm:centre_3}]
			This case is similar to the proof for type~\labelcref{itm:centre_2}:
			Again, we can assume without loss of generality for simplicity that for every $n \in \NN$ there exists $\beta(n) \in \NN$ such that the vertices $a_n, b_n$ are contained in the weak component of $A - P_{\beta(n)}$ that has finitely many vertices of $R$ and $a_{n+1}, b_{n+1}$ are contained in the weak component of $A - P_{\beta(n)}$ containing a tail of $R$.
			
			If there exists an infinite set $J \subseteq \NN$ such that $a_j, b_j$ are internal vertices of some $P_n$ with $b_j \leq_{P_n} a_j$ for every $j \in J$, we construct a butterfly minor shaped by a star similarly as in the proof for type~\labelcref{itm:centre_2}:
			By butterfly contracting $b_j P_n a_j$, we can assume $b_j = a_j$. Then we find arborescences $T^{\mathrm{in}}$ and $T^{\mathrm{out}}$ in $A$ rooted at $r$ that intersect only in $\{r\} \cup \{a_j = b_j: j \in J\}$.
			By butterfly contracting  $T^{\mathrm{in}}$ and $T^{\mathrm{out}}$, we obtain a butterfly minor shaped by a star with all teeth in $U$.
			
			Thus there exists an infinite subset $J \subseteq \NN$ such that $a_j, b_j$ are internal vertices of some directed path $P_n$ only if $a_j <_{P_n} b_j$.
			By butterfly contracting $P_n a_j$ for every $a_j$ in some $P_n$, we can assume that $a_j \in V(R)$ for every $j \in J$.
			If infinitely many elements of $(b_j: j \in J)$ are contained in $V(R)$, we find a butterfly minor shaped by a star with all teeth in $U$ by~\cref{lemhelp} and~\cref{obs:reversing}.
			Otherwise there exists an infinite subset $J' \subseteq J$ such that the vertices $(b_j: j \in J')$ are internal vertices of the directed paths $(P_n: n \in \NN)$.
			We find an in-arborescence $T^{\mathrm{in}}$ in $A$ rooted in $r$ that intersects $R$ only in $r$ and contains $(b_j: j \in J')$.
			By restricting to $R \cup T^{\mathrm{in}} \cup \bigcup_{j \in J'} (F_j + (a_j, x_j) + (x_j, b_j))$ and butterfly contracting $R$ and $T^{\mathrm{in}}$, we obtain the desired butterfly minor shaped by a star with all teeth in $U$.
			
			\item[If $A$ is of type~\labelcref{itm:centre_4}] Since all vertices of $(a_i: i \in \NN)$ are distinct and all vertices of $(b_i: i \in \NN)$ are distinct, there is an infinite subset $I \subseteq \NN$ such that for every $i < j \in I$ there exists $f(i,j) \in \NN$ with $a_i, b_i \in V(\bigcup_{m \in [f(i,j)]} C_m) \setminus V(C_{f(i,j)+1})$ and $a_j, b_j \in V(\bigcup_{m \in \NN \setminus [f(i,j)]} C_m) \setminus V(C_{f(i,j)})$.
			For simplicity, we assume without loss of generality that $I = \NN$ and set $\beta(n):=f(n, n+1)$.
			
			Let $D'''$ be the butterfly minor obtained from $D''$ by butterfly contracting the directed path $C_{\beta(n)} \cap C_{\beta(n) + 1}$ to a vertex $z_n$ for every $n \in \NN$.
			For every $n \in \NN$, let $D^n:= (F_n + (a_n,x_n) + (x_n, b_n)) \cup \bigcup_{m \in [f(n)]\setminus [f(n-1)]} C_m \subseteq D'''$, where $f(0):=0$.
			Note that $D''' = \bigcup_{n \in \NN} D^n$.
			Furthermore, $D^n \cap D^m = \emptyset$ for $|m -n| \geq 2$ and $D^n \cap D^{n+1} = \{z_n\}$ for every $n, m \in \NN$.
			
			For every $n \in 2 \NN + 1$ we take a butterfly minor $\hat{D}^n$ of $D^n$ that is a double path with endpoints $z_{n - 1}, z_n$, which exists by~\cref{prop:laced_paths,obs:double_path}.
			Furthermore, let $\hat{D}^1$ be a butterfly minor of $D^1$ that is a double path with endpoints $u_1$ and $z_1$, which exists by~\cref{prop:laced_paths,obs:double_path}.
  			For every $n \in 2 \NN$ we take a butterfly minor $\hat{D}^n$ of $D^n$ that is a double path containing $u_n, z_{n-1}, z_n$, the union of non-trivial double paths $F_1, F_2, F_3$ that intersect only in a common endpoint and whose other endpoints are $u_n, z_{n-1}, z_n$, or the union of a directed cycle $C$ of length $3$ and disjoint (possibly trivial) double paths $F_1, F_2, F_3$ having one endpoint in $V(C)$ and whose other endpoints are $u_n, z_{n-1}, z_n$, which exists by~\cref{prop3star}.
			See~\cref{fig:construct_comb}.
			Then there exists an infinite subset $K \subseteq 2 \NN$ such that either all elements of $(\hat{D}^k: k \in K)$ contain a directed cycle or no element of $(\hat{D}^k: k \in K)$ contains a directed cycle.
			In both cases, we contract $\hat{D}^n$ to a double path with endpoints $z_{n-1}$ and $z_n$ for every $n \in 2 \NN \setminus K$.
			Then $\bigcup_{n \in \NN} \hat{D}^n$ is the desired butterfly minor shaped by a comb.
			 \qedhere
		\end{description}

\begin{figure}[ht]
	\begin{tikzpicture}
		
		\draw[edge][] (0.25,0) to (0.75,0);
		\draw[edge] ({0.75},0) to ({0.25},0);
		\draw[edge] ({0.2},0.05) to ({0.2},0.95);
		\draw[edge] ({0.2},0.95) to ({0.2},0.05);
		\draw[edge][] (0.85,0) to (1.15,0);
		\draw[edge] ({1.15},0) to ({0.85},0);
		
		\draw[] ({0.8},0) node {.};
		\draw[] ({1.2},0) node {.};
		
		\draw[] ({0.2},0) node {.};
		\draw[LavenderMagenta] (0.2,1) node {\huge.};
		
		\foreach \a in {2,4,6}{

			\draw[LavenderMagenta] (\a,1) node {\huge.};

			\draw[edge] ({\a-0.15},0) to ({\a+0.15},0);
			\draw[edge] ({\a+0.15},0.05) to ({\a+0.05},0.35);
			\draw[edge] ({\a-0.05},0.35) to ({\a-0.15},0.05);

			\draw[edge] ({\a},0.45) to ({\a},0.95);
			\draw[edge] ({\a},0.95) to ({\a},0.45);
			
			\draw[] ({\a-0.2},0) node {.};
			\draw[] ({\a+0.2},0) node {.};
			\draw[] ({\a},0.4) node {.};	
		}
		
		\foreach \a in {1,...,5}{
			
			\draw[edge] ({\a+0.25},0) to ({\a+0.75},0);
			\draw[edge] ({\a+0.75},0) to ({\a+0.25},0);
		}
		
		\foreach \a in {3,5}{
			\draw[edge] (\a-0.15,0) to (\a + 0.15,0);
			\draw[edge] (\a+0.15,0) to (\a - 0.15,0);
			\draw[] (\a-0.2,0) node {.};
			\draw[] (\a+0.2,0) node {.};
		}
		
		\draw[edge] ({6.35},0) to ({6.25},0);
		\draw[edge,path fading=east] ({6.75},0) to ({6.25},0);
		
		\draw[decorate, decoration={brace,amplitude=5pt}] (1.2,-0.8) -- (0.2,-0.8) node[midway, below=3pt] {$\hat{D}^1$};
		
		\draw[decorate, decoration={brace,amplitude=5pt}] (2.2,-0.8) -- (1.3,-0.8) node[midway, below=3pt] {$\hat{D}^2$};
		
		\draw[decorate, decoration={brace,amplitude=5pt}] (3.7,-0.8) -- (2.3,-0.8) node[midway, below=3pt] {$\hat{D}^3$};
		
		\draw[decorate, decoration={brace,amplitude=5pt}] (4.2,-0.8) -- (3.8,-0.8) node[midway, below=3pt] {$\hat{D}^4$};
		
		\draw[decorate, decoration={brace,amplitude=5pt}] (5.2,-0.8) -- (4.3,-0.8) node[midway, below=3pt] {$\hat{D}^5$};
		
		\draw[decorate, decoration={brace,amplitude=5pt}] (6.2,-0.8) -- (5.3,-0.8) node[midway, below=3pt] {$\hat{D}^6$};
		
		\draw (1.25, -0.3) node {$z_1$};
		\draw (2.25, -0.3) node {$z_2$};
		\draw (3.75, -0.3) node {$z_3$};
		\draw (4.25, -0.3) node {$z_4$};
		\draw (5.25, -0.3) node {$z_5$};
		\draw (6.25, -0.3) node {$z_6$};

		\draw (0.2, 1.3) node {$u_1$};		
		\draw (2, 1.3) node {$u_2$};
		\draw (4, 1.3) node {$u_4$};
		\draw (6, 1.3) node {$u_6$};
		
	\end{tikzpicture}
		\caption{The butterfly minors $(\hat{D}^n)_{n \in \NN}$ if $A$ is of type~\labelcref{itm:centre_4} in the proof of~\cref{main_theorem}.}
				\label{fig:construct_comb}
\end{figure}

	\end{proof}

\section{The necessity of all types of directed graphs in~\cref{main_theorem}}\label{sec:necessity}
\cref{main_theorem} states six types of directed graphs, that is directed graphs shaped by a star whose centre is either a single vertex or a dominated directed ray, dominated directed rays, directed graphs shaped by a comb without directed cycles or with directed cycles, and chains of triangles.
\begin{prop}
	All six types of directed graphs in~\cref{main_theorem} are necessary for a statement like~\cref{main_theorem}.
\end{prop}
\begin{proof}
	Let $D$ be an arbitrary directed graph of one of these six types.
	Furthermore, let $D'$ be an arbitrary strongly connected butterfly minor of $D$ containing infinitely many teeth of $D$.
	We show that $D'$ is not of a different type than $D$.
	Let $\mu$ be a tree-like model of $D'$ in $D$.
	
	\begin{description}
		\item[If $D$ is a dominated directed ray]
		We show that $D'$ is a subdivision of a dominated directed ray with possibly some parallel edges.
		Thus $D'$ is not of a different type than $D$.
		
		Let $R$ be the ray of $D$ rooted in $r$ and let $e_v$ be the edges of $D$ between $r$ and $v$ for every $v \in V(R) \setminus \{r\}$.
		The subgraph $\mu(D')$ contains all edges of $R$ and infinitely many edges of $(e_v)_{v \in V(R)\setminus \{r\}}$ since $\mu(D')$ is strongly connected and infinite.
		For every edge $e_v$ that is present in $\mu(D')$ both endpoints of $e_v$ have in- and out-degree at least $1$ in $\mu(D') - e_v$.
		Thus no $e_v$ is contained in some branch set of $\mu$.
		This implies that all branch sets of $\mu$ are weakly connected subgraphs of $R$, and since $D'$ is infinite, they are subpaths of $R$.
		We can deduce that $D'$ is a subdivision of a dominated directed ray with possibly parallel edges.
			
		\item[If $D$ is shaped by a comb with a dominated directed ray in the centre]
		We show that $D'$ is the union of infinitely many disjoint double rays and a subdivision of a dominated directed ray with possibly some parallel edges.
		Then this implies that $D'$ is not of a different type than $D$.
		
		Let $R$ be the $r$-rooted ray of the dominated directed ray in $D$ and let $e_v$ be the edges between $r$ and $v$ for every $v \in V(R) \setminus \{r\}$.
		For every tooth $t$ of $D$ that $D'$ preserves, the double path connecting $t$ to the dominated directed ray is contained in $\mu(D')$ since $\mu(D')$ is strongly connected.
		Moreover, this double path is preserved in $D'$ since all its vertices either are $t$ or have either in- and out-degree at least $2$ in $\mu(D')$.
		
		The infinitely many disjoint double paths in $D'$ witness that infinitely many vertices of $R$ are preserved in $D'$.
		As in the previous case, none of the edges $(e_v)_{v \in V(R)\setminus \{r\}}$ is contained in a branch set of $\mu$, and thus all non-trivial branch sets of $\mu$ are subpaths of $R$.
		Then $D'$ is the union of infinitely many disjoint double rays and a subdivision of a dominated directed ray with possibly some parallel edges.
		
		\item[If $D$ is shaped by a star with a single vertex in the centre]
		For every tooth $t$ of $D$ that $D'$ preserves, the double path connecting $t$ to the centre vertex is contained in $\mu(D')$.
		Moreover, this double path is preserved in $D'$ since all its vertices either are $t$ or have either in- and out-degree at least $2$ in $\mu(D')$.
		This implies that there are infinitely many in-edges and infinitely many out-edges incident with the centre vertex in $D'$.
		Since directed graph of different types do not contain a vertex of infinite in-degree and infinite out-degree, $D'$ is not of a different type than $D$.
		
		\item[If $D$ is shaped by a comb without directed cycles]
		There there is an undirected ray $R$ rooted in some $v \in V(D)$ such that $\mu(D')$ contains $\mathcal{D}(R)$ since $\mu(D')$ is strongly connected.
		Since all vertices of $\mathcal{D}(R)$ but $v$ have in- and out-degree at least $2$ in $\mu(D')$, $\mathcal{D}(R) - v$ is contained in $D'$.
		Since no directed graph of a different type contains $\mathcal{D}(\hat{R})$ for some undirected ray $\hat{R}$, $D'$ is not of a different type than $D$.
		
		\item[If $D$ is shaped by a comb with directed cycles]
			For every tooth that $D'$ preserves, the corresponding directed cycle of length $3$ is contained in $\mu(D')$.
			For every but the first such directed cycle in $\mu(D')$, all its vertices have in- and out-degree $2$ in $\mu(D')$, and thus they are in distinct branch sets, which implies that the directed cycle is contained in $D'$.
			Thus $D'$ contains infinitely many directed cycles of length $3$.
			By~\cref{propconnectedness}, $D'$ does not contain two disjoint rays since $D$ does not have two disjoint rays, which implies that $D'$ is not a chain of triangles.
			We can deduce that $D'$ is not of a different type than $D$.
			
		\item[If $D$ is a chain of triangles]
			We show that $D'$ contains two disjoint rays, which implies that $D'$ is not of a different type than $D$.
			For simplicity, we assume without loss of generality that $D$ has the same vertex and edge labelling as in the introduction.
			See~\cref{fig:chain_of_triangles}.
			
			Since $\mu(D')$ is strongly connected, for every tooth $y_i$ that is preserved in $D'$ the edges $(b_i,c_i)$ and $(a_i, b_i)$ are contained in $\mu(D')$.
			Let $k \in \NN$ such that $y_k$ is preserved in $D'$.
			Then all edges $(c_i, c_{i+1})_{i \geq k}$ and $(a_{i+1}, a_i)_{i \geq k}$ are contained in $\mu(D')$ since $\mu(D')$ is strongly connected and contains infinitely many teeth of $D$.
			Moreover, infinitely many edges of $(c_j, a_j)_{i \geq k}$ are contained in $\mu(D')$.
			
			We construct an out-ray in $D'$ as follows:
			First, we show that infinitely many edges of $(c_i, c_{i+1})_{i \geq k}$ are not contained in branch sets of $\mu$.
			Let $\ell \geq k$ be arbitrary.
			Furthermore, let $\ell < i < j$ such that $(c_i,a_i)$ and $(b_j,c_j)$ are present in $\mu(D')$.
			Since $c_i$ has out-degree $2$ in $\mu(D')$ and $c_j$ has in-degree $2$ in $\mu(D')$, $c_i$ is contained in some out-branching and $c_j$ is contained in some in-branching of $\mu$.
			Then some edge $e \in ((c_m, c_{m +1}))_{i \leq m < j}$ is not contained in a branch set of $\mu$.
			Since $\ell$ was chosen arbitrarily, infinitely many edges of $((c_m, c_{m+1}))_{m \geq k}$ are not contained in branch sets.
			
			Let $\alpha: \NN \mapsto \NN_{\geq k}$ be the strictly increasing sequence such that $((c_{\alpha(n)},c_{\alpha(n)+1}))_{n \in \NN}$ contains precisely the edges of $((c_i, c_{i+1}))_{i \geq k}$ that are not contained in branch sets of $\mu$.
			Let $n \in \NN$ be arbitrary.
			Then the path $\bigcup_{\alpha(n)+1 \leq i < \alpha(n+1)} (c_i, c_{i+1})$ is contained in the branch set $\mu(v)$ for some $v \in V(D')$ and starts in the in-branching and ends in the out-branching of $\mu(v)$, which is witnessed by the edges $(c_{\alpha(n)},c_{\alpha(n)+1})$ and $(c_{\alpha(n+1)},c_{\alpha(n+1)+1})$.
			Thus the root of $\mu(v)$ is contained in $(c_i)_{\alpha(n)+1 \leq i \leq \alpha(n+1)}$.
			We deduce that there exists an out-ray $R$ in $D'$ whose edges are contained in $((c_i, c_{i+1}))_{i \geq k}$ and such that for each $v \in V(R)$ the branch set $\mu(v)$ has its root in $(c_i)_{i \geq k}$.
			
			Similarly, we can construct an in-ray $R'$ in $D'$ whose edges are contained in $((a_{i+1}, a_{i}))_{i \geq k}$ and such that for each $v \in V(R)$ the branch set $\mu(v)$ has its root in $(a_i)_{i \geq k}$.
			Note that $R$ and $R'$ are disjoint in $D'$. \qedhere
	\end{description}	
\end{proof}

\bibliography{ref.bib}

\end{document}